\documentclass[default,12pt]{sn-jnl}
\usepackage[utf8]{inputenc}
\usepackage{hyperref}
\usepackage{verbatim}
\usepackage{amsmath,amssymb,amsfonts,amsthm}
\usepackage{natbib}
\usepackage{cancel}
\usepackage{url}
\usepackage{float}
\usepackage{graphicx}
\graphicspath{ {./images/} }
\usepackage{physics}
\usepackage{bm}
\usepackage{wasysym}
\usepackage{multirow}
\usepackage{mathrsfs}
\usepackage[title]{appendix}
\usepackage{xcolor}
\usepackage{textcomp}
\usepackage{manyfoot}
\usepackage{booktabs}
\usepackage{algorithm}
\usepackage{algorithmicx}
\usepackage{algpseudocode}
\usepackage{listings}
\usepackage{esint}
\DeclareMathOperator{\diver}{div}

\DeclareMathOperator{\IM}{Im}

\def\eps{\varepsilon}
\def\d{{\partial}}
\def\R{{\mathbb R}}

\numberwithin{equation}{section}

\theoremstyle{plain}
\newtheorem{thm}{Theorem}[section]
\newtheorem{defn}[thm]{Definition}
\newtheorem{lem}[thm]{Lemma}

\newtheorem{prop}[thm]{Proposition}
\newtheorem{notat}[thm]{Notation}
\theoremstyle{remark}
\newtheorem{rmk}[thm]{Remark}

\begin{document}

\title{The zero capillarity limit for the Euler-Korteweg system with no-flux boundary conditions}

\author*[1]{\fnm{Paolo} \sur{Antonelli}}\email{paolo.antonelli@gssi.it}
\author*[2]{\fnm{Yuri} \sur{Cacchiò }}\email{yuri.cacchio@univie.ac.at}

\affil[1]{\orgname{Gran Sasso Science Institute}, \orgaddress{\street{Viale Francesco Crispi, 7}, \city{L'Aquila}, \postcode{67100}, \state{Italy}}}

\affil[2]{\orgname{Faculty of Mathematics, University of Vienna}, \orgaddress{\street{Oskar-Morgenstern-Platz 1}, \city{Vienna}, \postcode{1090}, \state{Austria}}}

\maketitle

\begin{abstract}
\ In this article, we study the small dispersion limit of the Euler-Korteweg system in a domain with a smooth boundary and no-flux boundary conditions. We exploit a relative energy approach to study the convergence of finite energy weak solutions towards strong solutions to the compressible Euler system. Given the boundary conditions under consideration, our approach requires a correction for the limiting particle density, due to the appearance of a boundary layer. Unlike conditional result on the vanishing viscosity limit, our analysis does not require additional conditions on the lack of anomalous concentration of capillary energy. This is due to the fact that the boundary layer appearing in our context is weaker than the one formed in the vanishing viscosity limit. We believe this approach can be adapted to study similar singular limits involving non-trivial boundary conditions.
\end{abstract}

\tableofcontents

\section{Introduction}
We study the following Euler-Korteweg (EK) system,
\begin{equation}\label{EK}
\left\{
    \begin{array}{rl}
    \partial_t \rho+\diver (\rho u)&=0;\\
    \partial_t (\rho u)+\diver\left(\rho u\otimes u\right)+\nabla p(\rho)&=\varepsilon^2\rho\nabla\left(\diver(k(\rho)\nabla\rho)-\frac{1}{2}k'(\rho)|\nabla\rho|^2\right),
\end{array}\right.
\end{equation}
posed on a domain $\Omega\subset\R^d$ with a smooth boundary $\partial\Omega$, which may be either bounded or unbounded.
The model \eqref{EK} describes the time evolution of a compressible, inviscid flow subject to self-consistent capillary forces in terms of its particle density $\rho=\rho(t, x)$ and velocity field $u=u(t, x)$. 
For simplicity, we assume the pressure to satisfy a barotropic power-type equation of state, namely 
\begin{equation}\label{pressure}
    p(\rho)=\rho^\gamma \text{ with }\gamma>1.
\end{equation}
 The term on the right-hand side of the momentum equation represents the Korteweg stress tensor and accounts for capillary effects in the model \cite{1573105975105495040, MR775366}.

The capillarity coefficient $k=k(\rho)\geq0$ is given and we assume it satisfies the following equation of state 
\begin{equation}\label{capillarity}
    k(\rho)=c_\alpha\rho^\alpha, \qquad c_\alpha>0, \, \alpha\geq-1.
\end{equation}
The Korteweg term was originally introduced in \cite{1573105975105495040} to describe capillary effects in liquid-vapor mixtures with diffuse interfaces, and it was derived later in \cite{MR775366}. We also refer to \cite{MR2760987}, where the derivation is obtained without introducing the concept of interstitial working. Similar models also arise in different physical contexts. For instance, equations \eqref{EK} with $k(\rho)=\frac{1}{4\rho}$ play a prominent role in the description of quantum fluids such as Bose-Einstein condensates or superfluid Helium \cite{pitaevskii2016bose}. With this choice of the capillarity coefficient, equations \eqref{EK} are usually called the \emph{quantum hydrodynamics} (QHD) system. For a more detailed discussion on the physical motivations and the derivation of models, we refer to \cite{MR775366, MR2760987, MR3916773} for general EK systems, and to \cite{AM_K} for QHD.

Our main interest lies in the analysis of the vanishing capillarity limit $\eps\to0$ for a class of initial-boundary value problems (IBVP) associated to \eqref{EK}. More precisely, we complement system \eqref{EK} with initial data
\begin{equation}\label{ID}
\rho|_{t=0}=\rho_0,\qquad (\rho u)|_{t=0}=J_0,
\end{equation}
and no-flux boundary conditions,
\begin{equation}\label{BC}
u\cdot n|_{\partial\Omega}=0,\qquad\partial_n\rho|_{\partial\Omega}=0.
\end{equation}
We denote by $n$ the unit outward normal. 
In the case when $\Omega$ is unbounded, we assume $\rho\to0$ and $\rho u\to0$, as $|x|\to\infty$.

In many physical contexts, the constant $\eps>0$ appearing in front of the Korteweg stress tensor \eqref{Korteweg term} is relatively small (see for instance \cite{jamet}). Hence, it is relevant from the applications' point of view to investigate the limit when $\eps\to0$.

Formally, we expect the asymptotic dynamics to be governed by the compressible Euler system
\begin{equation}\label{Euler}
\left\{
    \begin{array}{rl}
    \partial_t  \rho^E +\diver ( \rho^E  u^E) &=0;\\
    \partial_t(  \rho^E  u^E) +\diver\left( \rho^E  u^E \otimes  u^E \right)+\nabla p( \rho^E ) &=0,
\end{array}\right.
\end{equation}
complemented with the following boundary conditions,
\begin{equation}\label{BC_E}
u^E\cdot n|_{\partial\Omega}=0.
\end{equation}
However, due to the mismatch of boundary conditions between \eqref{BC} and \eqref{BC_E}, the formation of a boundary layer is expected in the limit $\eps\to0$.

Our main result provides a rigorous justification of the above heuristics in several aspects. First of all, we show the convergence of weak solutions to \eqref{EK}-\eqref{BC} towards solutions to \eqref{Euler}-\eqref{BC_E} in the natural topology determined by the energy associated to the limiting system \eqref{Euler}. This is achieved by a relative energy argument. Let us notice that this first result is insensitive to the phenomenon of boundary layer formation mentioned before, as for the mass density it only provides convergence in suitable Lebesgue spaces. 

Given the passage from boundary conditions \eqref{BC} to \eqref{BC_E}, we need to investigate the convergence of the gradient of the particle density close to the boundary. This analysis is performed by considering a second relative energy functional, associated to the total energy for the original system \eqref{EK}. In particular, the contribution from the capillary energy may be interpreted as a kinetic energy associated to an auxiliary velocity field
\begin{equation*}
v=v(\rho)=\sqrt{c_\alpha}\nabla\rho^{\frac{\alpha+1}{2}},
\end{equation*}
see \eqref{eq:en_EK_kinetic} below.
In this way, the second boundary condition in \eqref{BC} may be formally interpreted as a no-flux condition on $v$, and its mismatch in the limiting dynamics requires the introduction of a boundary layer correction for $v(\rho^E)$. Roughly speaking, our approach is similar in spirit to Kato's seminal paper on the vanishing viscosity limit \cite{Kato}, see also \cite{MR3171346} for the compressible case. There are, however, some main differences. First, while in \cite{Kato} the singular limit involves conditions on the energy dissipation, in our case the original system \eqref{EK} is conservative. What we need to avoid is that the contribution from the capillary energy concentrates close to the boundary. However, while the size of the boundary layer in \cite{Kato} is determined by the relative energy argument, we do not need further conditions on the lack of capillary energy concentration close to the boundary. This reflects the fact that the zero capillarity limit generates a weaker boundary layer compared to the one appearing for vanishing viscosity limits.

\subsection{Survey of related results}\label{Sec survey of}
The Cauchy problem \eqref{EK}-\eqref{ID} posed on $\Omega=\R^d$ is known to be locally well-posed for $(\rho, u)\in (\bar\rho, \bar u)+H^{s+1}\times H^s$, with $s>\frac{d}{2}+1$, where $(\bar\rho, \bar u)\in \R_+\times \R^d$ is a reference homogeneous solution \cite{BGDD}. Global dispersive solutions may be obtained for irrotational flows under more restrictive assumptions on the initial data, see \cite{AH}. For general solutions, the time of existence is estimated in terms of the rotational part of the velocity field \cite{MR4401388}.

The problem is also studied on periodic domains \cite{BMM, FIM}, where it is possible to prove non-uniqueness results by relying upon convex integration techniques \cite{DFM_ill}. In this context, the vanishing capillarity limit is studied in \cite{MR3647257} by a relative energy argument, under specific assumptions on the pressure and the capillarity term. A similar relative-energy approach is also exploited to analyze the viscous counterpart of system \eqref{EK}, given by the so-called Navier-Stokes-Korteweg (NSK) system. In \cite{BGLV}, for instance, the authors construct dissipative solutions to \eqref{EK} as vanishing viscosity limit of global-in-time weak solutions to the NSK system. Moreover, the vanishing capillarity limit of NSK towards the compressible Navier-Stokes equations may also be proved by a compactness argument \cite{BDL}, owing to the extra dissipative estimates available in this case. Similar arguments yield results on different singular limits, see for instance \cite{ACLS, BCMNX, CCL, carnevale2025extending}.

The QHD system, obtained by choosing $k(\rho)=\frac{1}{4\rho}$ in \eqref{EK}, plays a special role in the theory, being related to the following nonlinear Schr\"odinger (NLS) equation
\begin{equation}\label{eq:NLS}
i\eps\d_t\psi=-\frac{\eps^2}{2}\Delta\psi+\frac{\gamma}{\gamma-1}|\psi|^{2(\gamma-1)}\psi,
\end{equation}
via the so-called Madelung transformations \cite{madelung1927quantentheorie}. More precisely, Madelung's approach derives system \eqref{EK} as the evolution equations for the momenta associated to $\psi$, by defining $\rho=|\psi|^2$, $\rho u=\eps\IM(\bar\psi\nabla\psi)$. For this reason, more results are available for the QHD than for general EK systems. In particular, it is possible to show the existence of global-in-time, finite energy weak solutions, see \cite{AM_CMP, AM_ARMA, AMZ_CMP} where the authors also establish a rigorous justification of the Madelung transformation for general, finite energy wave functions.
Notice that the hydrodynamic boundary conditions \eqref{BC} correspond to homogeneous Neumann conditions for the wave function, i.e.
\begin{equation}\label{BC_NLS}
\partial_n\psi|_{\partial\Omega}=0.
\end{equation}
Equation \eqref{eq:NLS} also highlights the relation between the vanishing capillarity limit for QHD and the semiclassical limit for the NLS equations, for which an extensive literature already exists \cite{CR,MR1425123,MR4274579}. Most of the existing literature considers the vanishing capillarity/semiclassical limit problem in the whole space. A restricted number of results are dedicated to domains with non-trivial boundaries. In \cite{LZ}, by a modulated energy approach, the semiclassical limit for \eqref{eq:NLS}-\eqref{BC_NLS} with $\gamma=2$ is proven in a two-dimensional exterior domain. The same problem posed on the half-space is considered in \cite{CR}, where the authors study the expansion of the solution $\psi^\eps$ in powers of $\eps$ by means of the WKB approach, calculating the boundary layer expansion. Their analysis shows that a weak boundary layer of spatial size $O(\eps)$ is formed. A similar result is also obtained with Dirichlet boundary conditions \cite{MR4379626}, where an extra smallness condition on initial and boundary data is needed in order to deal with the stronger boundary layer.

In the recent paper \cite{audiard2025zeroA}, the authors study the vanishing dispersion limit for the Euler-Korteweg system. In particular, they perform a semiclassical analysis for the QHD system on the half-plane with hybrid boundary conditions.

An alternative approach to study the vanishing capillarity limit for \eqref{EK} is to derive suitable time-decay properties for the linearized operator and show uniform estimates in the parameter $\eps>0$, see the interesting analysis performed in \cite{song2025global}. We also refer to \cite{MR3196944}, where the same problem is studied in the viscous setting.

Our strategy of proof exploits a relative energy argument. This is a well-established tool for studying PDEs, where it is used to prove results about stability \cite{Daf, dafermos1979stability}, weak-strong uniqueness \cite{MR523630, MR2992037}, and singular limits \cite{MR1748352, MR3594360, MR1944031, carnevale2025extending}. In our paper, we study a relative energy functional as in \cite{BGLV,MR3647257} to analyze the singular limit from solutions to \eqref{EK} towards solutions to \eqref{Euler}.

\subsection{Main contributions of the paper}
Our paper exploits a relative-energy method to show the convergence of finite energy weak solutions to \eqref{EK}-\eqref{BC} towards strong (classical) solutions to \eqref{Euler}-\eqref{BC_E}. A similar approach was considered in \cite{MR3647257} for the problem posed in a periodic domain, but their result holds under restrictive conditions on the capillarity coefficient and the pressure law. We allow for a more general class of Euler-Korteweg systems, and moreover, we are able to consider domain with non-trivial boundary conditions.

This is mainly due to the different structure of the relative energy functional we consider. In the higher order energy $\mathcal E_h$, that will be introduced in \eqref{eq:aug_rel} below, the capillary energy is interpreted as the kinetic energy associated to an auxiliary velocity field defined as in \eqref{def:v}. This is reminiscent of the augmented formulation of Euler-Korteweg and Navier-Stokes-Korteweg models \cite{MR3734208}. A similar relative energy was introduced in \cite{BGLV} to study the vanishing viscosity limit for the Navier-Stokes-Korteweg system, see also the recent preprint \cite{carnevale2025extending}.

By exploiting the Madelung transform, it is possible to interpret the results in \cite{LZ, CR} as a zero capillarity limit for \eqref{EK}-\eqref{BC} with $k(\rho)=\frac{1}{4\rho}$. The present analysis extends these results to a more general class of Euler-Korteweg systems. Our analysis requires less regularity on the initial data \eqref{ID} and applies to general finite energy weak solutions to \eqref{EK}. Moreover, by considering the extended relative energy $\mathcal E_h$, in Theorem \ref{main high} we are able to prove also the convergence of the mass density gradient by taking into account the correction due to the boundary layer formation. Our result yields rougher information than the full boundary layer expansion derived in \cite{CR}, however, the present analysis holds under more general assumptions, especially concerning the regularity of the solutions.

\section{Preliminaries and statement of the main results}
In this Section, we recall some preliminary results, introduce the basic notions needed throughout our paper and state the main results.

First of all, we notice that the Korteweg term may be written in divergence form as follows,
\begin{equation}\label{Korteweg term}
\varepsilon^2\rho\nabla\left(\diver(k(\rho)\nabla\rho)-\frac{1}{2}k'(\rho)|\nabla\rho|^2\right)
=\varepsilon^2\diver\mathbb K,
\end{equation}
where the Korteweg stress tensor $\mathbb K$ is given by
\begin{equation}\label{eq:K_stress}
\mathbb{K}=\left(\diver(\rho k(\rho)\nabla\rho)-\frac{1}{2}(\rho k'(\rho)+k(\rho)|\nabla\rho|^2)\right)\mathbb{I}-k(\rho)\nabla\rho\otimes\nabla\rho,
\end{equation}
and $\mathbb I$ is the identity matrix.

The total energy associated to the Euler-Korteweg system \eqref{EK} is defined by
\begin{equation}\label{eq:en_EK}
E_{EK}=\int_{\Omega}\frac12\rho|u|^2+f(\rho)
+\frac{\eps^2}{2}k(\rho)|\nabla\rho|^2\,dx.
\end{equation}
The first contribution in \eqref{eq:en_EK} is the kinetic energy, whereas the second one is given by the internal energy density, defined as
\begin{equation}\label{general case}
    f(\rho)=\rho\int^\rho_0\frac{p(s)}{s^2}ds.
\end{equation}
Equivalently, we also have $p(\rho)=\rho f'(\rho)-f(\rho)$.
Notice that in the particular case under our consideration $p(\rho)=\rho^\gamma$, with $\gamma>1$, we have
\begin{equation}\label{eq:int_en_d}
f(\rho)=\frac{1}{\gamma-1}\rho^\gamma.
\end{equation}
Finally, the last term in \eqref{eq:en_EK} accounts for the contribution from the capillary energy. The total energy $E_{EK}$ is formally conserved along the flow \eqref{EK}. No dissipative effects are encoded in system \eqref{EK} and the only natural a priori estimates on $(\rho, u)$ are the ones given by \eqref{eq:en_EK}.

In particular, the only available bound on $u$ is in $L^2(\Omega;\rho\,dx)$, which implies that no control on the velocity field is available in the vacuum region $\{\rho=0\}$. Since we consider global-in-time finite energy weak solutions, we are interested in allowing for the mass density to vanish. For this reason, it seems more convenient to study the system by considering the mass and momentum densities $\rho$ and $J=\rho u$, respectively. In this way, the total energy becomes
\begin{equation}\label{total energy}
E_{EK}=\int_{\Omega}\frac12\frac{|J|^2}{\rho}+\frac{1}{\gamma-1}\rho^\gamma
+\frac{\eps^2}{2}k(\rho)|\nabla\rho|^2\,dx.
\end{equation}
This approach is customary for inviscid compressible models such as the compressible Euler system, see \cite{feireisl2025well} for instance.
For $(\rho, J)$ to be of finite energy, we must have that $J=0$ a.e. on $\{\rho=0\}$. In fact, for the purpose of our analysis, it is convenient to denote the \emph{pseudovelocity} field
\begin{equation}\label{eq:Lambda}
\Lambda=\frac{J}{\sqrt{\rho}}=\sqrt{\rho}u,
\end{equation}
for which the energy bound yields $\Lambda\in L^\infty_tL^2_x$. The vector field $\Lambda$ does not have a clear physical interpretation. However, its notion turns out to be useful to provide a rigorous definition of finite energy weak solutions to \eqref{EK}.

Analogously, it is convenient to define the following auxiliary state functions,
\begin{align}
    K(\rho)&=\int_0^\rho sk(s)ds,\label{K}\\
    \beta(\rho)&=\int_0^\rho\sqrt{k(s)}ds.\label{beta}
\end{align}
In this way, from \eqref{eq:K_stress} we write the Korteweg term as follows
\begin{equation*}
\diver\mathbb K = \nabla\left(\Delta K(\rho)
-\frac12K''(\rho)|\nabla\rho|^2\right)-\diver\Big(\nabla\beta(\rho)\otimes\nabla\beta(\rho)\Big).
\end{equation*}
Our choice of the capillarity coefficient \eqref{capillarity} implies that
\begin{equation}\label{eq:cap_coeff}
\beta(\rho)=\frac{2\sqrt{c_\alpha}}{2+\alpha}\rho^{1+\frac\alpha2},\quad
K(\rho)=\frac{c_\alpha}{2+\alpha}\rho^{2+\alpha}.
\end{equation}
The capillary energy may be written as 
\begin{equation}
    \frac12\int_{\Omega}k(\rho)|\nabla\rho|^2\,dx=\frac12\int_{\Omega}|\nabla\beta(\rho)|^2\,dx.
\end{equation}
Thus, a finite energy state satisfies $\nabla\beta(\rho)\in L^2$. Moreover, our assumption \eqref{capillarity} implies that there exists a constant $\omega>0$ such that 
\begin{equation}\label{capillarity ineq}
    |\rho k'(\rho)|\leq\omega k(\rho),
\end{equation}
that in turn yields
\begin{equation}\label{eq:cap_bound}
|K''(\rho)||\nabla\rho|^2\lesssim k(\rho)|\nabla\rho|^2=|\nabla\beta(\rho)|^2\in L^1.
\end{equation}
Resuming, we rewrite the Euler-Korteweg system \eqref{EK} as follows,
\begin{equation}\label{EK_w}
\left\{\begin{aligned}
&\d_t\rho+\diver J=0\\
&\d_tJ+\diver\left(\Lambda\otimes\Lambda\right)+\nabla p(\rho)
=\eps^2\nabla\left(\Delta K(\rho)-\frac12K''(\rho)|\nabla\rho|^2\right)-\eps^2\diver(\nabla\beta(\rho)\otimes\nabla\beta(\rho)),
\end{aligned}\right.
\end{equation}
where $\Lambda$ is such that $J=\sqrt{\rho}\Lambda$. Notice that all terms in \eqref{EK_w} make sense in the distributional sense for finite energy states. Moreover, when testing system \eqref{EK_w} against suitably regular functions, we see that the boundary conditions \eqref{BC} are generalized as follows,
\begin{equation}\label{BC_w}
\d_n\beta(\rho)|_{\d\Omega}=0,\qquad\Lambda\cdot n|_{\d\Omega}=0.
\end{equation}
Again, this is dictated by the possible appearance of vacuum regions in the problem, which prevents to uniquely define the velocity field.
\begin{defn}[Weak solutions]\label{def: WS-EK}
Let $p(\rho),k(\rho)$ be given by \eqref{pressure}-\eqref{capillarity}, respectively, and let $T>0$. Let the initial data $(\rho_0, J_0)\in L^1(\Omega;\R^{1+d})$ be such that
\begin{equation*}
    E_0 = \int_\Omega \left( \frac12\frac{|J_0|^2}{\rho_0} + \frac{1}{\gamma-1}\rho_0^\gamma + \frac{\varepsilon^2}{2}k(\rho_0)|\nabla\rho_0|^2 \right) dx < \infty.
\end{equation*}
We say that $(\rho,J)$ is a \emph{finite energy weak solution} to the IBVP \eqref{EK_w}, \eqref{ID}, \eqref{BC_w} on $[0, T]\times\Omega$, if there exists a pair $(\sqrt{\rho},\Lambda)$ of functions such that 
\begin{enumerate}

    \item $\sqrt{\rho}\in L^\infty([0,T]; L^2(\Omega)\cap L^{2\gamma}(\Omega)) $, and $\Lambda\in L^\infty([0,T]; L^2(\Omega))$;
    \item by defining $\beta=\beta(\rho)$ and $K=K(\rho)$ as in \eqref{eq:cap_coeff}, we have 
    $\nabla\beta(\rho)\in L^\infty([0,T]; L^2(\Omega))$ and $\nabla K(\rho)\in L^1([0,T]\times\Omega)$;
    
    \item by defining $\rho=(\sqrt{\rho})^2$, $J=\sqrt{\rho}\Lambda$,  we have $\rho\in C_{weak}([0,T]; L^1(\Omega)\cap L^{\gamma}(\Omega)) $, $J \in C_{weak} ( [0,T] ; L^{2\gamma/(\gamma+1)}(\Omega))$ and the integral identities 
    \begin{align}
        \int_{\Omega} \rho\psi \,dx\bigg|_{t=0}^T&=\int_0^T\int_{\Omega}\rho\partial_t\psi+J\cdot\nabla\psi \,dxdt,\label{weak1}\\
        \int_{\Omega}J\cdot\phi \,dx\bigg|_{t=0}^T=\int_0^T\int_{\Omega}&J\cdot\partial_t\phi+\Lambda\otimes\Lambda:\nabla\phi+p(\rho)\diver\phi +\varepsilon^2\nabla\beta(\rho)\otimes\nabla\beta(\rho):\nabla\phi\label{weak2}\\+&\frac{\varepsilon^2}{2}K''(\rho)|\nabla\rho|^2\diver\phi +\varepsilon^2\nabla K(\rho)\cdot\nabla(\diver\phi) \,dxdt\notag.
    \end{align}
    hold for every pair of test functions $(\psi,\phi)\in C^1([0,T]; C^1_c(\overline{\Omega})\cross C^2_c(\overline{\Omega})^d)$ such that $\phi\cdot n|_{\partial\Omega}=0$.
    \item Moreover, the following energy inequality holds for a.e. $\tau\in(0,T)$,
    \begin{equation}\label{energy inequality}
    E_{EK}(\tau) = \int_{\Omega} \left(\frac{1}{2}|\Lambda|^2 + \frac{\varepsilon^2}{2}|\nabla\beta(\rho)|^2 + f(\rho) \right)(\tau,x)dx \leq E_0.
\end{equation}
\end{enumerate}
We say that $(\rho, J)$ is a \emph{global-in-time} finite energy weak solution, if the above definition holds for any arbitrary $0<T<\infty$.
\newline
Moreover, we say that $(\rho, J)$ is a \emph{dissipative weak solution} if, in addition to point (1) -- (4) above, we have 
\begin{equation}
    \frac{d}{dt}E_{EK}(t)\leq 0 \quad \text{in } \mathcal D'((0, T)).
\end{equation}
\end{defn}
\begin{rmk}
We emphasize that our notion of \emph{dissipative weak solution} refers to a finite energy weak solution in the sense of distributions that additionally satisfies the global energy dissipation inequality $\frac{d}{dt}E_{EK}(t) \leq 0$. This should not be confused with the weaker concept of ``dissipative solutions'' used to Euler-Korteweg and quantum fluid systems (see, e.g., \cite{MR3961293}).
\end{rmk}
We now consider the limiting dynamics, resulting in 
the Euler system \eqref{Euler}, 
equipped  with the boundary condition \eqref{BC_E}
and initial data
\begin{equation}\label{initial data E}
     \rho^E|_{t=0}=\rho^E _0, \quad u^E|_{t=0}=u^E _0.
\end{equation}
The local-in-time well-posedness of the compressible Euler system is extensively studied via the theory of symmetric hyperbolic systems. The following theorem recalls a standard existence result for local-in-time strong solutions, specifically adapted from \cite{Makino1985}. For a general overview of the classical well-posedness theory, we also refer the reader to \cite{MR616950, MR834481, MR623940, BenzoniSerre}.

\begin{thm}\label{theo sol euler}
    Let $\Omega\subset \mathbb{R}^d$ be an open set with smooth boundary $\partial \Omega$. Let $( \rho^E _0, u^E _0)\in H^{s}(\Omega)$ be the initial data, with $s>\frac{d}{2}+k$ and $k\geq 1$, such that the compatibility conditions $\partial_t^j u^E\cdot n|_{t=0,\partial\Omega} = 0$ hold for $ j = 0, \dots, s-1$. Assume that:
    \begin{enumerate}
        \item[(i)] either $1<\gamma\leq\frac{s+2}{s}$;
        \item[(ii)] or $\Omega$ is bounded, $1<\gamma<\infty$, and $\inf_{x\in\Omega}\rho_0^E>0$.
    \end{enumerate}
    Then, there exists $T^*>0$ and a unique classical solution 
    \begin{equation*}
        ( \rho^E , u^E )\in C([0,T^*];H^{s}(\Omega))\cap C^1([0,T^*];H^{s-1}(\Omega))
    \end{equation*}
    of \eqref{Euler}, \eqref{BC_E}, \eqref{initial data E}. In particular, by Sobolev embeddings, 
    \begin{equation}
        (\rho^E, u^E) \in C([0,T^*];C^k(\overline{\Omega})).
    \end{equation}
\end{thm}

\begin{rmk}
We point out a few important facts concerning the well-posedness framework adopted above:
\begin{itemize}
    \item \textbf{Absence of vacuum:} If the initial density is bounded away from zero, i.e., $\inf_{x\in\Omega}\rho_0^E>0$, then the upper bound on the index $\gamma$ is not required. The local existence of classical solutions holds for any $\gamma>1$, as the singularity of the transformation map at the vacuum state is never reached. Since the solution is continuous in time, the strict positivity is preserved on the interval of existence $[0, T^*]$.
    \item \textbf{Dimensionality:} Although the original result in \cite{Makino1985} was stated for $d=3$, the theorem extends naturally to any dimension $d \geq 1$. The symmetrization process relies on the algebraic structure of the Euler equations, which is dimension-independent, and the resulting system remains symmetric hyperbolic. 
    \item \textbf{Boundary conditions:} The no-flux boundary condition \eqref{BC_E} is maximal dissipative (conservative). Indeed, following \cite{Makino1985} for isentropic flows $p(\rho)=\rho^\gamma$, one introduces the variable $w=\rho^{(\gamma-1)/2}$ and rewrites \eqref{Euler} as $A_0(U)\partial_t U+\sum_{j=1}^d A_j(U)\partial_{x_j} U=0$, where $U=(w,u_1,...,u_d)$. The boundary matrix is defined as $A_n(U)=\sum_{j=1}^dA_j(U)n_j$. Evaluated at $\partial \Omega$, we have,
    \begin{equation*}
        A_n(U)|_{\partial \Omega}=\begin{pmatrix}
            0 & \frac{\gamma-1}{2}w n^t\\
            \frac{\gamma-1}{2}w n & 0
        \end{pmatrix}.
    \end{equation*}
    The energy at the boundary is controlled by the quadratic form $U^TA_n(U)U = (\gamma-1)w^2(u\cdot n)$, which vanishes on $\partial\Omega$ due to \eqref{BC_E}. Since the system is symmetric and the boundary condition is maximal dissipative, well-posedness follows from the classical theory for characteristic boundaries (see \cite[Chapter 3]{BenzoniSerre}), without the need to verify the general Kreiss-Lopatinskii condition.
\end{itemize}
\end{rmk}
Let us remark that for the solutions considered in Theorem \ref{theo sol euler}, by combining the two equations in \eqref{Euler}, the following evolutionary equation for the velocity field may be obtained for solutions satisfying $\rho^E>0$,
\begin{equation}\label{eq:Euler_v}
\partial_tu^E+u^E\cdot\nabla u^E+\nabla f'(\rho^E)=0,
\end{equation}
where $f'(\rho)$ is the derivative of the internal energy density with respect to its argument, see \eqref{general case}.

\subsection{Relative energy functionals}

We now introduce the relative energy functionals we are going to consider in this paper. The relative energy (or relative entropy) method is a well-established tool in the stability analysis and singular limits of partial differential equations, originating from classical works on conservation laws \cite{Daf, dafermos1979stability}. For recent applications of this approach specifically tailored to Euler-Korteweg and Navier-Stokes-Korteweg systems, we refer the reader to \cite{MR3961293, BGLV, MR3647257, MR3594360}.

Let $(\rho, J)$ be a finite energy weak solution to \eqref{EK} in the sense of Definition \ref{def: WS-EK}. Given a pair of sufficiently smooth test functions $(r, U)$, we define the relative energy functional as follows,
\begin{equation}\label{def:rel}
\mathcal E(\rho, J | r, U) = \int_{\Omega}\frac{1}{2}|\Lambda-\sqrt{\rho} U |^2+f(\rho| r )+\frac{\varepsilon^2}{2}|\nabla\beta(\rho)|^2dx
\end{equation}
where the relative internal energy density is defined by
\begin{equation}\label{eq:rel_int_en}
           f(\rho|r)=f(\rho)-f'( r )\cdot(\rho- r )-f( r ).
\end{equation}
Given our specific choice of the internal energy density, see \eqref{eq:int_en_d}, the relative internal energy density reads
\begin{equation*}
f(\rho|r)=\frac{1}{\gamma-1}\rho^\gamma+r^\gamma-\frac{\gamma}{\gamma-1}r^{\gamma-1}\rho.
\end{equation*}
By using the above definitions \eqref{total energy}, \eqref{def:rel}, \eqref{eq:rel_int_en} we check that the following identity holds
\begin{equation}\label{eq:rel_sp}\begin{aligned}
\mathcal E= E_{EK} -\int_{\Omega}J\cdot U+\rho\left(-\frac12|U|^2+f'(r)\right)\,dx +\int_{\Omega}p(r)\,dx.
\end{aligned}\end{equation}
where $p=p(r)$ is defined in \eqref{pressure}.

In the case of the QHD system with quadratic pressure, namely when $k(\rho)=\frac{1}{4\rho}$ and $\gamma=2$, we now show that, by exploiting the Madelung transformation, it is possible to establish an analogy between two functionals. The first is the functional 
\begin{equation}\label{eq:rel_GP}
    H^\eps(t)=\int_{\Omega}\frac12|(i\eps\nabla+U)\psi^\eps|^2+\frac12\left(|\psi^\eps|^2-\rho\right)^2\,dx,
\end{equation}
studied in \cite{LZ}, see formula (12) therein, for the semi-classical limit of the Gross-Pitaevskii equation. The second is the relative energy functional introduced in \eqref{def:rel}.
Notice that, by \eqref{eq:cap_coeff}, we have $\beta(\rho)=\sqrt{\rho}$.

To rigorously justify this analogy, we briefly recall the polar factorization approach that provides a rigorous justification of the Madelung transformation for arbitrary finite energy wave functions $\psi$. We address the reader to \cite[Section 1.5.1]{AM_K}, for a more comprehensive exposition.

For any $\psi\in H^1(\Omega)$, we define the set of polar factors as
\begin{equation*}
P(\psi)=\{\varphi\in L^\infty(\Omega)\,:\,\psi=|\psi|\varphi\quad\textrm{a.e. on}\;\Omega\}.
\end{equation*}
Since the mass and momentum densities associated to $\psi^\eps$ are given by
\begin{equation*}
\rho=|\psi^\eps|^2,\qquad J=\eps\IM(\bar\psi^\eps\nabla\psi^\eps),
\end{equation*}
respectively, and given the relation \eqref{eq:Lambda} between $J$ and the pseudovelocity $\Lambda$, we define
\begin{equation*}
\Lambda=\eps\IM(\bar\varphi^\eps\nabla\psi^\eps),
\end{equation*}
where $\varphi^\eps\in P(\psi^\eps)$.
\newline
The definition of the hydrodynamic state $(\sqrt{\rho}, \Lambda)=(|\psi|, \eps\IM(\bar\varphi\nabla\psi))$ achieved by means of the polar factorization allows to construct a finite energy weak solution to the QHD system, see \cite{AM_CMP, AM_ARMA, AMZ_CMP}. The main advantage of this approach is that the quantities $(\sqrt{\rho}, \Lambda)$ are defined a.e. in $\Omega$, thus avoiding the definition of the velocity field  $u$ that may become singular close to the vacuum region $\{\rho=0\}$.

\begin{prop}\label{lem:GP}
Let $\psi\in\mathcal C(\R; H^1(\Omega))$ be a (mild) solution to the NLS equation \eqref{eq:NLS} and let us define
\begin{equation*}
\sqrt{\rho}=|\psi|, \qquad\Lambda=\IM(\bar\varphi\nabla\psi),
\end{equation*}
where $\varphi\in P(\psi)$. Then $(\rho, J)=((\sqrt{\rho})^2, \sqrt{\rho}\Lambda)$ is a finite energy weak solution to the QHD system \eqref{EK} with $\alpha=-1$, $c_\alpha=\frac14$. Moreover, the following identity holds for any $t\in\R$,
\begin{equation*}
H^\eps=\mathcal E_{QHD}(\rho, J|r, U)=
\int_{\Omega}\frac12|\Lambda-\sqrt{\rho}U|^2+f(\rho|r)+\frac{\eps^2}{2}|\nabla\sqrt{\rho}|^2\,dx,
\end{equation*}
where $H^\eps$ is defined in \eqref{eq:rel_GP}.
\end{prop}
For further details, we also refer the reader to \cite[Chapter 3]{MR4274579}.\\
Our first main theorem provides a convergence result of finite energy weak solutions to \eqref{EK} towards strong solutions to \eqref{Euler}. This is achieved by deriving a Gr\"onwall-type estimate on the relative energy functional $\mathcal{E}$ defined in \eqref{def:rel}. In view of Proposition \ref{lem:GP} above, the following theorem may be seen as the analogue of the main result in \cite{LZ} for the QHD system. However, our analysis applies to a larger range of the capillarity coefficient and considers more general domains in arbitrary space dimensions. We emphasize that the rigorous derivation of the Gr\"onwall estimate requires the limiting velocity field to possess at least $\mathcal{C}^2$ spatial regularity, which determines the assumptions on the initial data for the Euler system \eqref{Euler}.

\begin{thm}\label{main}
Let $(\rho^\eps, J^\eps)$ be a dissipative weak solution to the IBVP \eqref{EK}, \eqref{ID}, \eqref{BC} in the sense of Definition \ref{def: WS-EK}. Let $(\rho^E, u^E)$ be a strong solution to \eqref{Euler}, \eqref{BC_E}, \eqref{initial data E} given by Theorem \ref{theo sol euler} such that $(\rho^E, u^E)\in C([0, T^*]; C^2(\overline{\Omega}))$, for some $T^*>0$. 
Let us further assume that the initial data satisfy
\begin{equation}\label{condizione sui dati iniziali}
\mathcal E(\rho_0^\varepsilon,J_0^\varepsilon| \rho^E _0, u^E _0) \to 0, \qquad \text{as } \eps \to 0.
\end{equation}
Then, as $\eps\to0$, we have
\begin{equation}\label{convergenza}
\norm{\rho^\varepsilon- \rho^E }_{L^\infty([0,T^*];L^1\cap L^\gamma(\Omega))}+\norm{\Lambda^\varepsilon-\sqrt{\rho^\varepsilon} u^E }_{L^\infty([0,T^*];L^2(\Omega))}\to 0.
\end{equation}
Consequently, the following convergence holds for the momentum density:
\begin{equation}\label{convergenza J}
\lim_{\eps\to0}\|J^\varepsilon -  \rho^E  u^E\|_{L^\infty([0, T^*]; L^1(\Omega))}=0.
\end{equation}
\end{thm}

\begin{rmk}\label{rmk:rate_of_convergence}
It is possible to obtain an explicit rate of convergence if we assume well-prepared initial data. Suppose that the initial states coincide, namely $\rho^\varepsilon_0 = \rho^E_0$ and $J^\varepsilon_0 = J^E_0$. Provided that the limiting initial density $\rho^E_0$ is sufficiently regular so that $\nabla\beta(\rho^E_0) \in L^2(\Omega)$, the initial relative energy is given by the initial capillary energy,
\begin{equation*}
\mathcal{E}(\rho^\varepsilon_0, J^\varepsilon_0 | \rho^E_0, u^E_0) = \frac{\varepsilon^2}{2} \int_{\Omega} |\nabla\beta(\rho^E_0)|^2 dx = O(\varepsilon^2).
\end{equation*}
Notice that the assumption $\nabla\beta(\rho^E_0) \in L^2(\Omega)$ is satisfied by the initial data fulfilling the assumption of Theorem \ref{theo sol euler}.

Following the estimates in Proposition \ref{stima sul resto}, the remainder term is bounded by $C(\mathcal{E} + \varepsilon^2)$. An application of Gr\"onwall's lemma yields $\mathcal{E}(\tau) = O(\varepsilon^2)$ for all $\tau \in [0, T^*]$. Consequently, this implies an explicit convergence rate of order $O(\varepsilon)$ for the kinetic term $\|\Lambda^\varepsilon - \sqrt{\rho^\varepsilon}u^E\|_{L^\infty([0,T^*]; L^2(\Omega))}$, and analogous rates for the density $\rho^\varepsilon$ relying on the lower bounds of the relative internal energy $f(\rho|\rho^E)$.
\end{rmk}

\begin{rmk}
We point out that Theorem \ref{main} does not contain any claim concerning the existence of weak solutions to the IBVP \eqref{EK_w}, \eqref{ID}, \eqref{BC_w}. To our knowledge, their existence is known only in the special case of QHD systems, i.e., $k(\rho)=\frac{1}{4\rho}$, and in the whole space, namely $\Omega=\R^d$ (see \cite{AM_CMP,AM_ARMA,AM_K} and Section \ref{Sec survey of} for surveyed results concerning also smooth solutions). The global-in-time existence of finite energy weak solutions for general Euler-Korteweg systems on domains with boundaries remains a challenging open problem and will be the subject of future investigations. 
\end{rmk}

Our first main result establishes the convergence of the standard hydrodynamic variables. However, we aim to improve this result by providing more detailed information about the convergence of the mass density gradient, especially close to the boundary. 

Indeed, in the standard relative energy $\mathcal{E}$ defined in \eqref{def:rel}, the kinetic and internal energy terms are expressed as distances between the exact solution and the limiting Euler solution. In contrast, the capillary energy $\frac{\varepsilon^2}{2}|\nabla\beta(\rho)|^2$ appears alone, lacking a reference counterpart from the limiting dynamics. To properly measure the distance between the gradients, we rewrite this term into a relative quantity.

To achieve this, we introduce a second relative energy functional that explicitly accounts for higher-order derivatives. The main idea is to interpret the capillary energy as the kinetic energy of an auxiliary velocity field. This approach is closely related to the augmented formulations for capillary fluids, such as those used by Bresch and coauthors (see e.g., \cite{MR3734208}).

By recalling formulas \eqref{K} and \eqref{beta}, we define the auxiliary velocity field $v$ as follows,
\begin{equation}\label{def:v}
    v(\rho)=\sqrt{\frac{k(\rho)}{\rho}}\nabla\rho =\frac{\nabla \mu(\rho)}{\rho}=\nabla\theta(\rho).
\end{equation}
Notice that, in this way, the capillary energy density may be interpreted as the kinetic energy associated to the velocity field $v(\rho)$, namely it is possible to write the total energy \eqref{eq:en_EK} as
\begin{equation}\label{eq:en_EK_kinetic}
    E_{EK} = \int_\Omega \left( \frac{1}{2}\frac{|J|^2}{\rho} + \frac{\varepsilon^2}{2}\rho |v|^2 + \frac{1}{\gamma-1}\rho^\gamma \right) dx.
\end{equation}
Moreover, we can define an associated auxiliary momentum density $m$ as
\begin{equation}\label{eq:aux_mom}
    m=\rho v(\rho)=\sqrt{\rho k(\rho)}\nabla\rho=\nabla\mu(\rho)=\sqrt{\rho} \nabla\beta(\rho).
\end{equation}
From definition \eqref{def:v}, the state functions $\mu=\mu(\rho)$ and $\theta=\theta(\rho)$ are determined by the relations,
\begin{equation}\label{def:theta}
    \mu'(\rho)=\sqrt{\rho k(\rho)}\quad \text{and} \quad \theta'(\rho)=\sqrt{\frac{k(\rho)}{\rho}}.
\end{equation}
We note that under our specific assumption for the capillarity coefficient \eqref{capillarity}, we have 
\begin{equation}
    \mu(\rho)=c_\alpha\frac{2}{3+\alpha}\rho^{(\alpha+3)/2},
\end{equation} which implies that $\sqrt{\rho}\nabla\mu'(\rho)\in L^2(\Omega)$. 

Notice that this integrability property holds true under the more general structural assumption \eqref{capillarity ineq}. Indeed, by expanding the derivative, we can bound it directly using the finite capillary energy:
\begin{equation}\label{eq:integrab_condition}
    \sqrt{\rho}\mu''(\rho)\nabla\rho=\frac{1}{2}\left(\sqrt{k(\rho)}+\frac{\rho k'(\rho)}{\sqrt{k(\rho)}}\right)\nabla\rho \lesssim \nabla \beta(\rho)\in L^2(\Omega).
\end{equation}
As a consequence of this general bound, we also obtain,
\begin{equation}
    m\cdot\nabla\mu'(\rho)=\nabla\beta(\rho)\cdot\big(\sqrt{\rho}\nabla\mu'(\rho)\big)\in L^1(\Omega).
\end{equation}
This auxiliary formulation allows us to rewrite the dispersive Korteweg term. Indeed, by using the definitions above, we have
\begin{equation}\label{riscrittura K}
    \Delta K(\rho)-\frac12K''(\rho)|\nabla\rho|^2=\diver (\mu'(\rho)m)-m\cdot\nabla\mu'(\rho)=\mu'(\rho)\diver m.
\end{equation}

By writing the Euler-Korteweg system in terms of these new variables, we can derive the evolution equations for our auxiliary fields.

\begin{lem}
    Let $(\rho,J)$ be a finite energy weak solution of the system \eqref{EK_w} in the sense of Definition \ref{def: WS-EK}. Then, the auxiliary momentum $m$ satisfies the following equation in the weak sense,
    \begin{equation}\label{eq: m}
        \partial_t m+\nabla\big(\diver\left(\mu'(\rho )J\right)-\Lambda\cdot (\sqrt{\rho}\nabla\mu'(\rho))\big)=0.
    \end{equation}
    Namely, for any test function $\varphi\in C([0,T]; C^2_c(\overline{\Omega})^d)$ such that $\varphi\cdot n|_{\d\Omega}=0$, we have,
    \begin{equation}\label{eq: weak v}
        \int_{\Omega} m\cdot\varphi \,dx\bigg|_{t=0}^T =\int_0^T\int_\Omega m\cdot\partial_t\varphi-\mu'(\rho) J\cdot \nabla\diver \varphi-\Lambda\cdot(\sqrt{\rho}\nabla\mu'(\rho))\diver \varphi \, dxdt.
    \end{equation}
    In addition, if $(\rho,J)$ is a strong solution of \eqref{EK_w} with strictly positive density $\rho>0$, the velocity fields $u$ and $v$ satisfy the following system,
    \begin{equation}\label{eq u e v}
    \left\{\begin{aligned}
        &\d_t u+ (u\cdot\nabla u)+\nabla\left(f'(\rho)-\varepsilon^2\mu'(\rho)\diver v-\frac{\varepsilon^2}{2}|v|^2\right)=0,\\
        &\d_t v+\nabla(u\cdot v+\mu'(\rho)\diver u)=0.
    \end{aligned}
    \right.
    \end{equation}
\end{lem}

\begin{rmk}We emphasize that the augmented system \eqref{eq u e v} is formally valid even in the vanishing capillarity limit $\varepsilon=0$. In this case, which corresponds to the Euler system for the limiting variables $(\rho^E, u^E)$, the auxiliary velocity $v^E = v(\rho^E)$ is driven by the velocity field $u^E$. This occurs because the dispersive terms $\varepsilon^2\nabla\left(\mu'(\rho^E)\diver v^E+\frac{1}{2}|v^E|^2\right)$ vanish.
\end{rmk}

\begin{proof}
    By multiplying the continuity equation by $\mu'(\rho)$, we obtain
    \begin{align*}
        0&=\d_t\mu(\rho)+\mu'(\rho)\diver J=\partial_t\mu(\rho)+\diver(\mu'(\rho)J)-J\cdot\nabla\mu'(\rho)\\
        &=\d_t\mu(\rho)+\diver(\mu'(\rho)J)-\Lambda\cdot \sqrt{\rho }\mu''(\rho)\nabla \rho.
    \end{align*}
Notice that all terms in the identity above are well-defined in the distributional sense, due to the integrability condition $\sqrt{\rho}\mu''(\rho)\nabla\rho \lesssim \nabla\beta(\rho) \in L^2(\Omega)$ established in \eqref{eq:integrab_condition}. Equation \eqref{eq: m} is then obtained by taking the gradient of the identity above.
    Similarly, multiplying the continuity equation by $\theta'(\rho)$ yields
    \begin{equation*}
        \d_t\theta+u\cdot\nabla \theta(\rho)+\rho\theta'(\rho)\diver u=0.
    \end{equation*}
    Taking the gradient of this identity provides the second equation in \eqref{eq u e v}. Finally, we rewrite the momentum equation as
    \begin{equation*}
        \rho\left(\d_t u+u\cdot\nabla u+\nabla f'(\rho)\right)=\varepsilon^2\rho\nabla\left(\diver\left(k(\rho)\nabla\rho\right)-\frac{1}{2}\rho k'(\rho)|\nabla\rho|^2\right).
    \end{equation*}
    Dividing by $\rho$ and noting that, by definition \eqref{def:v}, we have
    \begin{align*}
        \diver (k(\rho)\nabla\rho)&=\diver(\sqrt{\rho k(\rho)} v)=\sqrt{\rho k(\rho)}\diver v+v\cdot\nabla (\sqrt{\rho k(\rho)})\\
        &=\sqrt{\rho k(\rho)}\diver v+\frac{1}{2}|v|^2+\frac{1}{2}k'(\rho)|\nabla \rho|^2,
    \end{align*}
    we recover the first equation in \eqref{eq u e v}.    
\end{proof}

In this context, we introduce the \emph{higher-order relative energy} $\mathcal{E}_h$. This functional adds the relative kinetic energy of our auxiliary velocity field to the standard relative energy. Given a finite energy weak solution $(\rho, J)$ and a triplet of sufficiently smooth test functions $(r, U, V)$, we define,
\begin{align}
        \mathcal{E}_h(\rho, \Lambda | r, U, V)&=\frac{1}{2}\int_\Omega\left(|\Lambda-\sqrt{\rho}U|^2+\varepsilon^2|\nabla \beta(\rho)-\sqrt{\rho}V|^2\right)+\int_{\Omega}f(\rho|r)\label{eq:aug_rel}\\
        &=E_{EK}-\int_\Omega J\cdot U+\rho\left(-\frac{1}{2}|U|^2-\frac{\varepsilon^2}{2}|V|^2+\varepsilon^2v\cdot V+f'(r)\right)+\int_\Omega p(r),\notag
\end{align}
where the last identity follows analogously to \eqref{eq:rel_sp}.

\begin{rmk}
By comparing the definition of the standard relative energy \eqref{def:rel} and the higher-order one \eqref{eq:aug_rel}, we see that they are linked by the following relation,
\begin{equation}\label{eq:entr_rel}
\mathcal{E}_h(\tau)=\mathcal{E}(\tau)+\int_\Omega \left( \frac{\varepsilon^2}{2}\rho |V|^2-\varepsilon^2m\cdot V \right) dx.
\end{equation}
\end{rmk}

We are now ready to state our second main result, which improves Theorem \ref{main} by proving the convergence of the gradient $\varepsilon\nabla\beta(\rho^\varepsilon)$ in $L^2$. 

To do this via a Gr\"onwall estimate on $\mathcal{E}_h$, we want to use the limiting auxiliary velocity $v(\rho^E)$ as our test function $V$ in \eqref{eq:aug_rel}. However, in general the density $\rho^E$ does not satisfy the homogeneous Neumann boundary conditions required by the original system, meaning that $v(\rho^E)\cdot n|_{\partial \Omega} \neq 0$. 

To resolve this mismatch, we must introduce a \emph{boundary layer correction} $v_{bl}$. By evaluating the energy against a corrected test function $v^E_{bl} = v(\rho^E) - v_{bl}$, we force the right boundary condition. This approach is inspired by Kato's seminal paper on the vanishing viscosity limit for the Navier-Stokes equations \cite{Kato}. However, there is a crucial difference: in the viscous case, friction on the boundary creates a boundary layer, and closing the estimates requires assuming a priori that the anomalous energy dissipation at the boundary vanishes. In our setting, the zero-capillarity limit generates a \emph{weaker} boundary layer. In fact, our relative energy approach absorbs the boundary layer errors, and we do not need to impose further assumptions on the capillary energy near the boundary. See also \cite{MR3537458,MR3171346} for similar constraints in the compressible case.

\begin{thm}\label{main high}
Let $(\rho^\eps, J^\eps)$ be a finite energy weak solution to \eqref{EK}, \eqref{ID}, \eqref{BC} in the sense of Definition \ref{def: WS-EK}. Let $(\rho^E, u^E)$ be the strong solution to \eqref{Euler} given by Theorem \ref{theo sol euler} corresponding to initial data with regularity $k \geq 3$, so that $(\rho^E, u^E)\in C([0, T^*]; C^3(\overline{\Omega}))$.
Let us further assume that the initial data satisfy
\begin{equation}\label{eq:high_ini_data}
\mathcal{E}_h(\rho_0^\varepsilon,\Lambda_0^\varepsilon | \rho^E _0, u^E _0, v^E_{0})\to 0, \qquad \text{as } \eps \to 0.
\end{equation}
Then, as $\eps\to0$ we have
\begin{align}
\norm{\rho^\varepsilon- \rho^E }_{L^\infty([0,T^*];L^1\cap L^\gamma(\Omega))}&+\norm{\Lambda^\varepsilon-\sqrt{\rho^\varepsilon} u^E }_{L^\infty([0,T^*];L^2(\Omega))}\notag\\
&+\eps\|\nabla\beta(\rho^\eps)-\sqrt{\rho^\varepsilon} v^E\|_{L^\infty([0, T^*];L^2(\Omega))}\to 0.
\end{align}
\end{thm}
%\begin{rmk}
%Notice that in the hypothesis of Theorem \ref{main high}, the initial relative energy is evaluated against the limiting velocity $v^E_0$, rather than the boundary-layer corrected profile $v^E_{bl, 0}$. This choice is rigorous. In fact, the boundary layer correction $v_{bl}$ is a technical tool introduced to deal the boundary mismatch during the time evolution. As shown in Section \ref{sec main high}, the difference between the relative energy evaluated with $v^E$ and $v^E_{bl}$ is controlled by the norm of the correction itself, leaving an error of order $O(\varepsilon^{2-2s})$. Consequently, if the initial relative energy vanishes using $v^E_0$, it vanishes for $v^E_{bl, 0}$ as well.
%\end{rmk}
The organization of the paper is as follows. In Section \ref{Sec remark} we recall some preliminary 
results. In Section \ref{Sec energy ineq}, we derive an energy inequality which provides a measure of distance between the weak solution of Euler-Korteweg equations in the sense of Definition \ref{def: WS-EK} and the strong solution of the Euler system in the sense of Theorem \ref{theo sol euler} in terms of distance of initial data plus a remainder term. Moreover, we show that this rest is controlled by the functional \eqref{def:rel}. Section \ref{Sec proof} is devoted to the conclusion of the proof of Theorem \ref{main}. Along the same flow, we derive preliminary estimates for the higher order relative entropy functional in Section \ref{sec high order}. Since $\rho^E$ does not satisfy the boundary condition \eqref{BC}, we introduce a boundary layer correction in Section \ref{sec kato correction} and we derive an energy estimate similarly to Section \ref{Sec energy ineq}. We make once again estimates on the remainder term in Section \ref{sec estimate R high}. Finally, we prove Theorem \ref{main high} in Section \ref{sec main high}.

\section{Proof of the general relative energy result}
\subsection{Preliminary results} \label{Sec remark}
Before we address the core of the paper, let us fix some notations and recall some well-known results. Throughout the paper, we omit the superscript $\varepsilon$ whenever it does not cause any confusion. 

We recall the definition of relative internal energy,
\begin{align}
    f(\rho| r )&=f(\rho)-f'( r )\cdot(\rho- r )-f( r )\notag\\
    &=\frac{1}{\gamma-1}\rho^\gamma-\frac{\gamma}{\gamma-1}r^{\gamma-1}\rho+r^{\gamma}.
\end{align}
It is straightforward to check that $\rho\mapsto f(\rho)$ is strictly convex, non-negative, and such that $f(\rho|r)=0$ only when $\rho=r$.

In the next lemma, we collect some Gagliardo-Nirenberg-type estimates for $\rho$ that will be used later in this section. Although elementary, we report their proof for the sake of completeness.
\begin{lem}\label{lemma: GN inequality}
    Let $\Omega \subset \mathbb{R}^d$ be a domain, and let $\rho: \Omega \to [0, \infty)$ be a sufficiently regular function such that $\rho \in L^1(\Omega)$ and $\nabla\beta(\rho) \in L^2(\Omega)$. 
    Then, the following inequalities hold:
    \begin{equation}
        \norm{\rho^{\frac{2+\alpha}{2}}}_{L^2}\lesssim \|\rho\|_{L^1}^{a}\|\nabla\beta(\rho)\|_{L^2}^{b},
    \end{equation}
    where
    \begin{itemize}
        \item If $d=1$, then $a=\frac{2+\alpha}{3+\alpha}$ and $b=\frac{1+\alpha}{3+\alpha}$.
        \item If $d=2$, then $a=\frac{1}{2}$ and $b=\frac{1+\alpha}{2+\alpha}$.
        \item If $d\geq3$, then $a=\frac{2+\alpha}{d(1+\alpha)+2}$ and $b=\frac{d(1+\alpha)}{d(1+\alpha)+2}$.
    \end{itemize}
  In particular, we have
    \begin{equation}\label{eq:inparticular}
        \eps^{b}\norm{\rho^{\frac{2+\alpha}{2}}}_{L^2}\lesssim \mathcal{E}^{\frac{b}{2}},
    \end{equation}
    where $\mathcal{E}$ is the relative energy defined in \eqref{def:rel}.
\end{lem}

\begin{proof}
    \noindent \textbf{Case $d=1$:} 
    By exploiting the inequality
    \begin{equation*}
    \|\rho\|_{L^\infty}^{\frac{\alpha}{2}+\frac32}\leq C\|\rho\|_{L^1}^{\frac12} \|\nabla\beta(\rho)\|_{L^2},
    \end{equation*}
    we infer that
    \begin{equation*}
    \|\rho\|_{L^\infty}\leq C\|\rho\|_{L^1}^{\frac{1}{\alpha+3}} \|\nabla\beta(\rho)\|^\frac{2}{\alpha+3}_{L^2}.
    \end{equation*}
    Consequently, 
    \begin{equation*}
    \left(\int\rho^{2+\alpha}\,dx\right)^{\frac{1}{2}}
    \leq\|\rho\|_{L^\infty}^{\frac{1+\alpha}{2}} \norm{\rho}_{L^1}^{\frac{1}{2}}\leq C\|\rho\|_{L^1}^{\frac{\alpha+2}{\alpha+3}} \|\nabla\beta(\rho)\|^\frac{\alpha+1}{\alpha+3}_{L^2}.
    \end{equation*}

    \noindent \textbf{Case $d=2$:}
    Let $p\geq2$ and let us define $q=\frac{2+\alpha}{2}p>1$ (recall that we are assuming $\alpha>-1$). Then, by interpolation, we have
    \begin{equation*}
    \|\rho\|_{L^{2+\alpha}}\leq\|\rho\|_{L^1}^{1-\theta}\|\rho\|_{L^q}^{\theta},
    \end{equation*}
    where $\theta=\frac{(1+\alpha)p}{2(q-1)}$. Notice that $\alpha\geq-1$ and $p\geq2$ imply that $\theta\in[0, 1]$. Moreover, by the Gagliardo-Nirenberg inequality we also have that
    \begin{equation*}
    \|\rho^{1+\frac{\alpha}{2}}\|_{L^p}\lesssim
    \|\rho^{1+\frac{\alpha}{2}}\|_{L^2}^{\frac 2p}\|\nabla\beta(\rho)\|_{L^2}^{\frac{p-2}{p}},
    \end{equation*}
    that is,
    \begin{equation*}
    \|\rho\|_{L^q}^{1+\frac{\alpha}{2}}\lesssim
    \|\rho\|_{L^{2+\alpha}}^{\frac{2+\alpha}{p}}\|\nabla\beta(\rho)\|_{L^2}^{\frac{p-2}{p}}.
    \end{equation*}
    Consequently, 
    \begin{equation*}
    \|\rho\|_{L^{2+\alpha}}^{\frac{2+\alpha}{2\theta}}
    \lesssim\| \rho\|_{L^1}^{\frac{(1-\theta)(2+\alpha)}{2\theta}}
    \|\rho\|_{L^{2+\alpha}}^{\frac{2+\alpha}{p}}
    \|\nabla\beta(\rho)\|_{L^2}^{\frac{p-2}{p}},
    \end{equation*}
    which in turn implies that
    \begin{equation*}
    \|\rho\|_{L^{2+\alpha}}^{\frac{2+\alpha}{2\theta}-\frac{2+\alpha}{p}}\lesssim
    \| \rho\|_{L^1}^{\frac{(1-\theta)(2+\alpha)}{2\theta}}\| \nabla\beta(\rho)\|_{L^2}^{\frac{p-2}{p}}.
    \end{equation*}
    Since
    \begin{equation*}
    \frac{2+\alpha}{2\theta}-\frac{2+\alpha}{p}=\frac{(p-2)(2+\alpha)^2}{2p(1+\alpha)},
    \end{equation*}
    we conclude
    \begin{align*}
    \|\rho\|_{L^{2+\alpha}}^{\frac{2+\alpha}{2}}&\lesssim
    \|\rho\|_{L^1}^{\frac{(1-\theta)}{2\theta}\frac{(1+\alpha)p}{p-2}}\|\nabla\beta(\rho)\|_{L^2}^{\frac{1+\alpha}{2+\alpha}}\\
    &=
    \|\rho\|_{L^1}^{\frac{1}{2}}\|\nabla\beta(\rho)\|_{L^2}^{\frac{1+\alpha}{2+\alpha}}.
    \end{align*}

    \noindent \textbf{Case $d \geq 3$:}
    Using Sobolev embedding:
    \begin{equation*}
        \norm{\rho}_{L^{\frac{d(2+\alpha)}{d-2}}}^{1+\frac{\alpha}{2}}=\norm{\rho}^{1+\frac{\alpha}{2}}_{L^{2^*\left(1+\frac{\alpha}{2}\right)}}\lesssim\norm{\nabla\beta(\rho)}_{L^2}.
    \end{equation*}
    Interpolation: let $\vartheta=\frac{d(1+\alpha)}{d(1+\alpha)+2}$,
    \begin{equation*}
        \norm{\rho}_{L^{2+\alpha}}\leq\norm{\rho}^{1-\vartheta}_{L^1}\norm{\rho}^\vartheta_{L^{\frac{d(2+\alpha)}{d-2}}}.
    \end{equation*}
    Then,
    \begin{align*}
    \norm{\rho}_{L^{2+\alpha}}^{\frac{2+\alpha}{2}}&\leq\norm{\rho}_{L^1}^{\frac{(2+\alpha)(1-\vartheta)}{2}}\norm{\rho}_{L^{\frac{d(2+\alpha)}{d-2}}}^{\vartheta\frac{2+\alpha}{2}}\\
    &\leq\norm{\rho}_{L^1}^{\frac{(2+\alpha)(1-\vartheta)}{2}}\norm{\nabla \beta(\rho)}_{L^2}^{\vartheta}\\
    &=\norm{\rho}_{L^1}^{\frac{2+\alpha}{d(1+\alpha)+2}}\norm{\nabla \beta(\rho)}_{L^2}^{\frac{d(1+\alpha)}{d(1+\alpha)+2}}.
    \end{align*}
Finally, to prove \eqref{eq:inparticular}, we observe that the relative energy $\mathcal{E}$ bounds the capillary energy, $\eps\|\nabla\beta(\rho)\|_{L^2} \lesssim \mathcal{E}^{\frac{1}{2}}$. Using this, we obtain
    \begin{equation*}
        \norm{\rho^{\frac{2+\alpha}{2}}}_{L^2} \lesssim \|\rho\|_{L^1}^{a} \|\nabla\beta(\rho)\|_{L^2}^{b} \lesssim \left(\eps^{-1} \mathcal{E}^{\frac{1}{2}}\right)^{b} = \eps^{-b} \mathcal{E}^{\frac{b}{2}}.
    \end{equation*}
    This concludes the proof.
\end{proof}

\subsection{Relative energy inequality}
\label{Sec energy ineq}
We now derive an energy inequality for the relative energy $\mathcal{E}$. This inequality is obtained in a canonical way by exploiting the weak formulations \eqref{weak1}, \eqref{weak2}, and the energy inequality \eqref{energy inequality} (see, for instance, \cite{MR2992037, MR3647257} for analogous computations). 

We recall that $(\rho, J)$ denotes a dissipative weak solution to \eqref{EK_w}-\eqref{BC_w}, defined through $\Lambda\in L^\infty_t L^2_x$ by $J=\sqrt{\rho}\Lambda$. Moreover, $(\rho^E, u^E)$ denotes the strong solution to the Euler system \eqref{Euler}-\eqref{BC_E} given by Theorem \ref{theo sol euler}.

\begin{lem}\label{lem:rel1}
Let $\mathcal E$ be defined as in \eqref{def:rel}, where $(\rho,J)$ is a dissipative weak solution to \eqref{EK_w}-\eqref{BC_w}, and $( \rho^E , u^E )\in C([0,T];C^1(\overline{\Omega})\cross C^2(\overline{\Omega})^d)$ is a strong solution to \eqref{Euler}-\eqref{BC_E}. Then, for a.e. $\tau\in[0, T]$, we have the following relative energy inequality,
    \begin{equation}\label{relative entropy}
        \mathcal{E}(\tau)\leq  \mathcal{E}(0)+\int_0^\tau \mathcal{R}_{rel}(\tau) dt,
    \end{equation}
where the remainder term $\mathcal{R}_{rel}(\tau) = \mathcal{R}_{rel}(\rho, \Lambda | \rho^E, u^E)(\tau)$ is defined as,
\begin{equation}\label{eq:R_rel_def}
\begin{aligned}
    \mathcal{R}_{rel}(\tau)=&-
    \int_\Omega (\Lambda-\sqrt{\rho} u^E ) \otimes(\Lambda-\sqrt{\rho} u^E ) :\nabla u^E \, dx \\
    &-\int_\Omega\diver u^E \left(p(\rho)-p( \rho^E )-p'( \rho^E )(\rho- \rho^E )\right) \, dx \\
    &-\int_\Omega\varepsilon^2\nabla\beta(\rho)\otimes\nabla\beta(\rho):\nabla u^E \, dx \\
    &-\int_\Omega\frac{\varepsilon^2}{2}K''(\rho)|\nabla\rho|^2\diver u^E \, dx \\
    &-\int_\Omega\varepsilon^2\nabla K(\rho)\cdot \nabla(\diver u^E ) \, dx.
\end{aligned}
\end{equation}
\end{lem}

\begin{proof}
Let us recall identity \eqref{eq:rel_sp} for the relative energy functional,
\begin{equation*}
\mathcal E(\tau) = E_{EK}(\tau) -\int_{\Omega} \left[J\cdot u^E+\rho\left(-\frac12|u^E|^2+f'(\rho^E)\right) \right]dx +\int_{\Omega}p(\rho^E)\,dx.
\end{equation*}

This structure suggests exploiting the weak formulation of the continuity and momentum equations \eqref{weak1}-\eqref{weak2}, evaluating them against the test functions $-\frac12|u^E|^2+f'(\rho^E)$ and $u^E$, respectively. Notice that, since $u^E$ solves the IBVP associated with the Euler system, it satisfies $u^E\cdot n|_{\partial\Omega}=0$, and hence it is an admissible test function for \eqref{weak2}.

For the continuity equation, identity \eqref{weak1} with $\psi=-\frac12|u^E|^2+f'(\rho^E)$ yields
\begin{equation*}\begin{aligned}
0=&\int_{\Omega}\rho\left(-\frac12|u^E|^2+f'(\rho^E)\right)\,dx\bigg|_{t=0}^\tau\\
&-\int_0^\tau\int_{\Omega} \rho\partial_t\left(-\frac12|u^E|^2+f'(\rho^E)\right) +J\cdot\nabla\left(-\frac12|u^E|^2+f'(\rho^E)\right)\,dxdt\\
=&\int_{\Omega}\rho\left(-\frac12|u^E|^2+f'(\rho^E)\right)\,dx\bigg|_{t=0}^\tau\\
&+\int_0^\tau\int_{\Omega}\left( \rho u^E\cdot\partial_tu^E + J^k u^{E, j}\partial_k u^{E, j} \right)\,dxdt -\int_0^\tau\int_{\Omega}\left( \rho\partial_tf'(\rho^E)+J\cdot\nabla f'(\rho^E) \right)\,dxdt.
\end{aligned}\end{equation*}

For the momentum equation, identity \eqref{weak2} evaluated against $\phi=u^E$ yields:
\begin{equation*}\begin{aligned}
0=&\int_{\Omega}J\cdot u^E\,dx\bigg|_{t=0}^\tau
-\int_0^\tau\int_{\Omega} \left(J\cdot\partial_tu^E +\big(\Lambda\otimes\Lambda+\varepsilon^2\nabla\beta(\rho)\otimes\nabla\beta(\rho)\big):\nabla u^E \right)\,dxdt\\
&-\int_0^\tau\int_{\Omega}\left( \big(p(\rho)+\varepsilon^2K''(\rho)|\nabla\rho|^2\big)\diver u^E +\varepsilon^2\nabla K(\rho)\cdot\nabla\diver u^E \right)\,dxdt.
\end{aligned}\end{equation*}
Summing up the two integral identities, we obtain
\begin{equation}\label{eq:weak_sum}\begin{aligned}
0=& \int_{\Omega} \left[ J\cdot u^E+\rho\left(-\frac12|u^E|^2+f'(\rho^E)\right) \right]\,dx\bigg|_{t=0}^\tau\\
&-\int_0^\tau\int_{\Omega} (J-\rho u^E)\cdot\partial_tu^E + \big(\Lambda^j\Lambda^k-J^k u^{E, j}\big)\partial_k u^{E, j}\,dxdt\\
&-\int_0^\tau\int_{\Omega} \left( p(\rho)\diver u^E+\rho\partial_tf'(\rho^E) +J\cdot\nabla f'(\rho^E) \right)\,dxdt\\
&-\varepsilon^2\int_0^\tau\int_{\Omega} \left( \nabla\beta(\rho)\otimes\nabla\beta(\rho):\nabla u^E +\frac12K''(\rho)|\nabla\rho|^2\diver u^E +\nabla K(\rho)\cdot\nabla\diver u^E \right)\,dxdt\\
=& \int_{\Omega} \left[ J\cdot u^E+\rho\left(-\frac12|u^E|^2+f'(\rho^E)\right) \right]\,dx\bigg|_{t=0}^\tau\\
&-\int_0^\tau\int_{\Omega}(J-\rho u^E)\cdot\big[-u^E\cdot\nabla u^E-\nabla f'(\rho^E)\big]+ \big(\Lambda^j\Lambda^k-J^k u^{E, j}\big)\partial_k u^{E, j}\,dxdt\\
&-\int_0^\tau\int_{\Omega}\left( p(\rho)\diver u^E+\rho\partial_tf'(\rho^E) +J\cdot\nabla f'(\rho^E) \right)\,dxdt\\
&-\varepsilon^2\int_0^\tau\int_{\Omega} \left( \nabla\beta(\rho)\otimes\nabla\beta(\rho):\nabla u^E +\frac12K''(\rho)|\nabla\rho|^2\diver u^E +\nabla K(\rho)\cdot\nabla\diver u^E \right)\,dxdt,
\end{aligned}\end{equation}
where in the last identity we used that $u^E$ solves \eqref{eq:Euler_v}.
By exploiting identities \eqref{eq:rel_sp}, \eqref{eq:weak_sum}, and recalling that the total energy $E_{EK}(\tau)$ is non-increasing for dissipative weak solutions, we deduce

\begin{equation}\label{eq:rel_en_c}\begin{aligned}
\mathcal E(\tau)=&E_{EK}(\tau)-\int_{\Omega}\left[J\cdot u^E+\rho\left(-\frac12|u^E|^2+f'(\rho^E)\right)+p(\rho^E)\right]dx\Bigg|_{t=\tau}\\
\leq &E_{EK}(0)-\int_{\Omega}\left[J\cdot u^E+\rho\left(-\frac12|u^E|^2+f'(\rho^E)\right)+p(\rho^E)\right]dx\Bigg|_{t=0}\\
&-\int_0^\tau\int_{\Omega}(J-\rho u^E)\cdot\left[-u^E\cdot\nabla u^E- \nabla f'(\rho^E)\right] +\big(\Lambda^j\Lambda^k- J^k u^{E, j}\big)\partial_k u^{E, j}\,dxdt\\
&-\int_0^\tau\int_{\Omega}\left[p(\rho)\diver u^E+\rho\partial_tf'(\rho^E) +J\cdot\nabla f'(\rho^E)-\partial_tp(\rho^E)\right]dxdt\\
&-\varepsilon^2\int_0^\tau\int_{\Omega}\left[\nabla\beta(\rho)\otimes\nabla\beta(\rho):\nabla u^E +\frac12K''(\rho)|\nabla\rho|^2\diver u^E+\nabla K(\rho)\cdot\nabla\diver u^E\right]dxdt\\
=& \mathcal E(0)
-\int_0^\tau\int_{\Omega}\left[ - \big(J-\rho u^E\big)^j u^{E, k}\partial_k u^{E, j} +\big(\Lambda^j\Lambda^k- J^k u^{E, j}\big)\partial_k u^{E, j}\right]dxdt\\
&-\int_0^\tau\int_{\Omega}\left[(p(\rho)-p(\rho^E))\diver u^E +p(\rho^E)\diver u^E+\rho u^E\cdot\nabla f'(\rho^E)+\rho\partial_tf'(\rho^E) -\partial_tp(\rho^E)\right]dxdt\\
&-\varepsilon^2\int_0^\tau\int_{\Omega}\left[\nabla\beta(\rho)\otimes\nabla\beta(\rho):\nabla u^E+\frac12K''(\rho)|\nabla\rho|^2\diver u^E+\nabla K(\rho)\cdot\nabla\diver u^E\right]dxdt.
\end{aligned}\end{equation}
Now, let us notice that the integral in the first line may be written
\begin{equation*}
 -\int_0^\tau\int_{\Omega}(\Lambda-\sqrt{\rho}u^E)\otimes(\Lambda-\sqrt{\rho}u^E):\nabla u^E\,dxdt.
\end{equation*}
Moreover, we write the third term in the second line of \eqref{eq:rel_en_c} as
\begin{equation*}
-\int_0^\tau\int_{\Omega}\rho u^E\cdot\nabla f'(\rho^E)\,dxdt
=-\int_0^\tau\int_{\Omega} \left[ (\rho-\rho^E)u^E\cdot\nabla f'(\rho^E)+u^E\cdot\nabla p(\rho^E) \right] \,dxdt,
\end{equation*}
where we used that $p'(\rho)=\rho f''(\rho)$. Since $u^E\cdot n|_{\partial\Omega}=0$, the integral above, combined with the term $-p(\rho^E)\diver u^E$ in \eqref{eq:rel_en_c}, simplifies to
\begin{equation*}
-\int_0^\tau\int_{\Omega}(\rho-\rho^E)u^E\cdot\nabla f'(\rho^E)\,dxdt.
\end{equation*}
Analogously, by noting that $\partial_t p(\rho^E)=\rho^E\partial_tf'(\rho^E)$, the remaining time derivatives in the second line of \eqref{eq:rel_en_c} become
\begin{equation*}
-\int_0^\tau\int_{\Omega}(\rho-\rho^E)\partial_tf'(\rho^E)\,dxdt.
\end{equation*}
By summing these contributions, the entire second line of \eqref{eq:rel_en_c} may be written as
\begin{equation}\label{eq:rel_en_d}\begin{aligned}
-\int_0^\tau\int_{\Omega}(p(\rho)-p(\rho^E))\diver u^E\,dxdt -\int_0^\tau\int_{\Omega}(\rho-\rho^E)[\partial_tf'(\rho^E)+u^E\cdot\nabla f'(\rho^E)]\,dxdt.
\end{aligned}\end{equation}
We now use the continuity equation for $\rho^E$ to infer that
\begin{equation*}
\d_tf'(\rho^E)+u^E\cdot\nabla f'(\rho^E)=-\rho^Ef''(\rho^E)\diver u^E.
\end{equation*}
By recalling again that $\rho^Ef''(\rho^E)=p'(\rho^E)$, this implies that \eqref{eq:rel_en_d} reconstructs the relative pressure
\begin{equation}\label{eq; pressure}
-\int_0^\tau\int_{\Omega}[p(\rho)-p(\rho^E)-(\rho-\rho^E)p'(\rho^E)]\diver u^E\,dxdt.
\end{equation}
Collecting all these simplifications back into \eqref{eq:rel_en_c}, we finally arrive at
\begin{equation*}\begin{aligned}
\mathcal E(\tau)\leq&\mathcal E(0)
-\int_0^\tau\int_{\Omega}(\Lambda-\sqrt{\rho}u^E)\otimes(\Lambda-\sqrt{\rho}u^E):\nabla u^E \,dxdt \\
&-\int_0^\tau\int_{\Omega} \big[p(\rho)-p(\rho^E)-(\rho-\rho^E)p'(\rho^E)\big]\diver u^E\,dxdt\\
&-\varepsilon^2\int_0^\tau\int_{\Omega}\left( \nabla\beta(\rho)\otimes\nabla\beta(\rho):\nabla u^E +\frac12K''(\rho)|\nabla\rho|^2\diver u^E+\nabla K(\rho)\cdot\nabla\diver u^E \right)\,dxdt,
\end{aligned}\end{equation*}
which proves the statement of the Lemma.
\end{proof}
In order to obtain a Gronwall-type estimate for the relative energy, we need to control the remainder term $\mathcal{R}_{rel}$ appearing in \eqref{relative entropy}. 

\begin{prop}\label{stima sul resto}
Under the same assumption of Lemma \ref{lem:rel1} and for a.e. $\tau \in[0,T]$, there holds
    \begin{equation}\label{eq:rem}
        |\mathcal{R}_{rel}(\tau)|\leq C\big(\mathcal{E}(\tau) +\varepsilon^2\big),
    \end{equation}
    where $C>0$ is a constant depending only on the smooth Euler solution and the initial energy.
\end{prop}

\begin{proof}
Let us recall the definition of $\mathcal{R}_{rel}$ given in \eqref{eq:R_rel_def}. Since $u^E$ is a strong solution on bounded time intervals, $u^E$ and its derivatives up to second order are uniformly bounded in $L^\infty(\Omega)$. 

The first four terms in $\mathcal{R}_{rel}$ are controlled by the relative energy $\mathcal{E}(\tau)$. Specifically,
\begin{align}
    &\int_\Omega |(\Lambda-\sqrt{\rho} u^E )\otimes(\Lambda-\sqrt{\rho} u^E ):\nabla u^E |\,dx \leq \norm{\nabla u^E }_{L^\infty}\int_{\Omega}|\Lambda-\sqrt{\rho} u^E |^2\,dx \leq C \mathcal{E}(\tau), \label{eq:bound_R1}\\
    &\int_{\Omega} |p(\rho)-p(\rho^E)-p'(\rho^E)(\rho-\rho^E)||\diver  u^E |\,dx \leq (\gamma-1)\norm{\diver u^E }_{L^\infty} \int_{\Omega}f(\rho| \rho^E )\,dx \leq C \mathcal{E}(\tau), \label{eq:bound_R2}\\
    &\int_\Omega\varepsilon^2|\nabla\beta(\rho)\otimes\nabla\beta(\rho):\nabla u^E |\,dx \leq \varepsilon^2 \norm{\nabla  u^E }_{L^\infty}\int_{\Omega} |\nabla\beta(\rho)|^2 \,dx \leq C\mathcal{E}(\tau). \label{eq:bound_R3}
\end{align}
Regarding the fourth term, by recalling the bound \eqref{eq:cap_bound} on the capillary coefficient, we have $|K''(\rho)||\nabla\rho|^2 \lesssim |\nabla\beta(\rho)|^2$. Thus,
\begin{equation}\label{eq:bound_R4}
    \int_\Omega \frac{\varepsilon^2}{2}|K''(\rho)||\nabla\rho|^2|\diver u^E| \,dx \lesssim \varepsilon^2 \norm{\diver u^E}_{L^\infty} \int_\Omega |\nabla\beta(\rho)|^2\,dx \leq C\mathcal{E}(\tau).
\end{equation}
The only term requiring further attention is the last one in \eqref{eq:R_rel_def}, involving the second derivative of the Euler velocity. By recalling identities \eqref{K}, \eqref{beta}, and \eqref{eq:cap_coeff}, we have $\nabla K(\rho) = \rho k(\rho)\nabla\rho = \rho^{1+\alpha/2}\nabla\beta(\rho)$. Consequently, using the Cauchy-Schwarz inequality, we deduce
\begin{align}
    \varepsilon^2\int_\Omega|\nabla K(\rho)\cdot \nabla(\diver u^E )| \,dx &\leq \varepsilon^2\norm{\nabla\diver u^E }_{L^\infty}\int_\Omega \rho^{1+\alpha/2}|\nabla\beta(\rho)| \,dx \notag\\
    &\leq C \varepsilon \norm{\rho^{1+\alpha/2}}_{L^{2}} \left(\varepsilon \norm{\nabla \beta(\rho)}_{L^2}\right)\notag\\
    &\leq C \varepsilon \norm{\rho^{1+\alpha/2}}_{L^{2}} \mathcal{E}(\tau)^{1/2}. \label{eq:bound_R5_CS}
\end{align}
Finally, we note that by the inequality established in Lemma \ref{lemma: GN inequality},  we have
\begin{equation}\label{eq:bound_R5_final}
    C \varepsilon \norm{\rho^{1+\alpha/2}}_{L^{2}} \mathcal{E}(\tau)^{1/2} \leq C \varepsilon^{1-b} \mathcal{E}(\tau)^{\frac{b+1}{2}}.
\end{equation}
An application of Young's inequality then yields the upper bound $C(\mathcal{E}(\tau) + \varepsilon^2)$. Collecting estimates \eqref{eq:bound_R1}-\eqref{eq:bound_R4} and this last bound concludes the proof.
\end{proof}

\subsection{Proof of Theorem \ref{main}}\label{Sec proof}

To conclude the proof, we first state and prove a general analytical lemma that guarantees strong convergence on unbounded domains.

\begin{lem}\label{lem:convergenza 1_gamma}
    Let $\Omega \subseteq \mathbb{R}^d$ and $\gamma>1$. For any $\varepsilon\in(0,1)$, let $\rho^\varepsilon: \Omega \to [0, \infty)$ be a sequence of measurable functions, and let $r: \Omega \to [0, \infty)$ be a smooth and bounded function. Assume that:
    \begin{enumerate}
        \item \emph{(Uniform tightness):} For any $\eta>0$, there exists $R=R(\eta)>0$ such that for all $\varepsilon \in (0,1)$,
        \begin{equation*}
            \int_{\Omega \setminus B_R(0)} \rho^\varepsilon(x)\,dx \leq \eta \quad \text{and} \quad \int_{\Omega \setminus B_R(0)} r(x)\,dx \leq \eta.
        \end{equation*}
        \item \emph{(Relative internal energy vanishing):} $\int_{\Omega}f(\rho^\varepsilon|r)\,dx\to 0$ as $\varepsilon\to 0$.
        \item \emph{(Uniform bound):} The sequence $\{\rho^\varepsilon\}$ is uniformly bounded in $L^\gamma(\Omega)$.
    \end{enumerate}
    Then, $\rho^\varepsilon \to r$ strongly in $L^1(\Omega) \cap L^\gamma(\Omega)$.
\end{lem}

\begin{proof}
    From assumption (2) and Chebyshev's inequality, we know that for every $\sigma_1>0$, 
    \begin{equation*}
        \big|\big\{x\in \Omega \mid f(\rho^\varepsilon |r) > \sigma_1 \big\}\big|\to 0 \quad \text{as } \varepsilon\to 0.
    \end{equation*}
    Let us fix $\sigma_2>0$. Since the relative internal energy $f(\cdot|r)$ is strictly convex and $r$ is globally bounded, there exists a uniform constant $\sigma_1>0$ (depending only on $\sigma_2$ and $\|r\|_{L^\infty}$) such that 
    \begin{equation*}
        \big\{x\in \Omega \mid |\rho^\varepsilon-r|>\sigma_2\big\} \subset \big\{x\in \Omega \mid f(\rho^\varepsilon |r)>\sigma_1\big\}.
    \end{equation*}
    Consequently,
    \begin{equation*}
        \big|\big\{x\in \Omega \mid |\rho^\varepsilon-r|>\sigma_2\big\}\big|\to 0 \quad \text{as } \varepsilon\to 0,
    \end{equation*}
    which implies that $\rho^\varepsilon$ converges to $r$ in measure. Because the sequence is uniformly tight (Assumption 1) and uniformly bounded in $L^\gamma(\Omega)$ (Assumption 3, which ensures uniform integrability), Vitali's Convergence Theorem guarantees that $\rho^\varepsilon \to r$ strongly in $L^1(\Omega)$. 
    
    Finally, to obtain the strong convergence in $L^\gamma(\Omega)$, we split the integral and exploit the bounds of the relative internal energy. Specifically, since $r$ is bounded, there exists a uniform constant $C>0$ such that $f(\rho|r) \ge C |\rho-r|^\gamma$ whenever $|\rho-r| > 1$. Therefore,
    \begin{equation*}
        \int_\Omega |\rho^\varepsilon - r|^\gamma \,dx = \int_{\{|\rho^\varepsilon - r| \le 1\}} |\rho^\varepsilon - r|^\gamma \,dx + \int_{\{|\rho^\varepsilon - r| > 1\}} |\rho^\varepsilon - r|^\gamma \,dx.
    \end{equation*}
    In the first integral on the right-hand side, since $\gamma > 1$, we have $|\rho^\varepsilon - r|^\gamma \le |\rho^\varepsilon - r|$. Thus, this term vanishes as $\varepsilon \to 0$ due to the strong $L^1$ convergence just established. The second integral is bounded by $C^{-1} \int_\Omega f(\rho^\varepsilon|r)\,dx$, which vanishes by Assumption 2. This concludes the proof.
\end{proof}

\begin{proof}[Conclusion of the Proof of Theorem \ref{main}]
By injecting the remainder estimate \eqref{eq:rem} provided by Proposition \ref{stima sul resto} into the relative energy inequality \eqref{relative entropy}, we obtain
\begin{equation*}
\mathcal E(\tau)\leq\mathcal E(0)+C\int_0^\tau \big( \mathcal E(t)+\varepsilon^2 \big)\,dt.
\end{equation*}
An application of Grönwall's Lemma yields
\begin{equation}\label{eq:Gronwall_final}
\sup_{\tau\in[0, T]}\mathcal E(\tau)\leq e^{CT}\left(\mathcal E(0)+C\varepsilon^2 T\right).
\end{equation}
By the assumptions on the initial data \eqref{condizione sui dati iniziali}, we know that $\mathcal E(0) \to 0$ as $\varepsilon \to 0$. Consequently, the right-hand side of \eqref{eq:Gronwall_final} vanishes, meaning
\begin{equation*}
\lim_{\varepsilon\to0}\sup_{\tau\in[0, T]}\mathcal E(\tau)=0.
\end{equation*}
Recalling the definition of the relative energy \eqref{def:rel}, this convergence implies,
\begin{align}
    &\lim_{\varepsilon\to 0} \|\Lambda-\sqrt{\rho}u^E\|_{L^\infty(0, T; L^2(\Omega))} = 0, \label{eq:conv_lambda}\\
    &\lim_{\varepsilon\to 0} \sup_{\tau\in[0, T]} \int_\Omega f(\rho|\rho^E)\,dx = 0. \label{eq:conv_f}
\end{align}
To show the strong convergence of the mass density \eqref{convergenza}, we apply Lemma \ref{lem:convergenza 1_gamma} with $r = \rho^E(t)$. The relative energy vanishing is given by \eqref{eq:conv_f}, and the uniform tightness and $L^\gamma$ bounds are guaranteed by the definition of finite energy weak solutions \eqref{def: WS-EK}. Thus, we deduce,
\begin{equation*}
    \rho^\varepsilon \to \rho^E \quad \text{strongly in } L^\infty(0,T; L^1(\Omega) \cap L^\gamma(\Omega)).
\end{equation*}
Finally, to prove the convergence of the momentum $J$ as stated in \eqref{convergenza J}, we decompose the difference $J - \rho^E u^E$ by adding and subtracting $\rho u^E$:
\begin{equation*}
    J - \rho^E u^E = (J - \rho u^E) + (\rho - \rho^E) u^E = \sqrt{\rho}(\Lambda - \sqrt{\rho}u^E) + (\rho - \rho^E)u^E.
\end{equation*}
We bound its $L^1$-norm using the triangle inequality and the Cauchy-Schwarz inequality, as follows
\begin{equation*}
    \| J - \rho^E u^E \|_{L^1(\Omega)} \leq \|\sqrt{\rho}\|_{L^2(\Omega)} \|\Lambda - \sqrt{\rho}u^E\|_{L^2(\Omega)} + \|\rho - \rho^E\|_{L^1(\Omega)} \|u^E\|_{L^\infty(\Omega)}.
\end{equation*}
Since the total mass $\|\rho\|_{L^1} = \|\sqrt{\rho}\|^2_{L^2}$ is uniformly bounded, $u^E$ is smooth, and both $\|\Lambda - \sqrt{\rho}u^E\|_{L^2}$ and $\|\rho - \rho^E\|_{L^1}$ vanish uniformly in time as $\varepsilon \to 0$, we conclude that,
\begin{equation*}
    J^\varepsilon \to \rho^E u^E \quad \text{strongly in } L^\infty(0,T; L^1(\Omega)).
\end{equation*}
\end{proof}

\section{The Augmented Relative Energy Inequality}\label{sec high order}
In this section, we consider the augmented relative energy $\mathcal E_h$ defined in \eqref{eq:aug_rel}. The primary motivation for introducing this higher-order functional is to control the density gradient $\nabla\rho$ and prove its strong convergence. 

For the sake of clarity, we first derive the augmented relative energy inequality evaluated against generic test functions $(r, U, V)$. Indeed, to perform the boundary layer analysis, our ultimate test functions are the Euler solutions modified by a boundary layer correction. These corrected profiles do not solve the limit system exactly, but generate a remainder. Establishing the identity for arbitrary test functions allows us to isolate these error terms before evaluating on the approximate solution.

\begin{lem}\label{lem:rel_hi}
 Let $(\rho,J)$ be a dissipative weak solution to \eqref{EK_w}-\eqref{BC_w} in the sense of Definition \ref{def: WS-EK}. Let the test functions $(r, U, V)$ belong to the spaces,
 \begin{equation*}
     r \in C^1([0,T]; C^1(\overline{\Omega})), \qquad U, V \in C^1([0,T]; C^2(\overline{\Omega})^d),
 \end{equation*}
 with $r$ uniformly bounded away from zero. Furthermore, assume that the vector fields satisfy the boundary conditions $U\cdot n|_{\partial\Omega}=V\cdot n|_{\partial\Omega}=0$. 
 Then, for a.e. $\tau \in [0,T]$, the following inequality holds:
\begin{align}\label{eq:aug_rel_b}
    \mathcal{E}_h(\tau) \leq \mathcal{E}_h(0) 
    &+ \int_0^\tau\int_\Omega (\rho U-J)\cdot\partial_t U + \big(J^k U^j-\Lambda^k\Lambda^j\big)\partial_k U^j \,dxdt \notag\\
    &- \int_0^\tau\int_\Omega p(\rho)\diver U + (\rho-r)\partial_t f'(r) + J\cdot\nabla f'(r) \,dxdt \notag\\
    &- \varepsilon^2 \int_0^\tau\int_\Omega \nabla\beta(\rho)\otimes\nabla\beta(\rho):\nabla U + \frac{1}{2}K''(\rho)|\nabla\rho|^2\diver U + \nabla K(\rho)\cdot\nabla(\diver U) \,dxdt \notag\\
    &+ \varepsilon^2 \int_0^\tau\int_\Omega (\rho V-m)\cdot\partial_t V + J^k V^j \partial_k V^j \,dxdt \notag\\
    &+ \varepsilon^2 \int_0^\tau\int_\Omega \mu'(\rho)J\cdot \nabla\diver V - \Lambda\cdot \big(\sqrt{\rho}\mu''(\rho)\nabla\rho\big)\diver V \,dxdt.
\end{align}
\end{lem}

\begin{proof}
The proof is straightforward and analogous to Lemma \ref{lem:rel1}. For the sake of clarity, we briefly sketch the main steps.

Proceeding as in the derivation of Lemma \ref{lem:rel1}, we evaluate the weak formulations \eqref{weak1} and \eqref{weak2} against the generic test functions $(r, U)$ instead of the exact Euler solution. This yields the following relative energy bound
\begin{equation}\label{eq:general_weak}
\begin{aligned}
\mathcal E(\tau)\leq&\mathcal E(0)
+\int_0^\tau\int_\Omega (\rho U - J)\cdot\d_t U + \big(J^k U^j - \Lambda^k\Lambda^j\big)\d_k U^j \,dxdt\\
&-\int_0^\tau\int_{\Omega} p(\rho)\diver U + (\rho-r)\d_t f'(r) + J\cdot\nabla f'(r)\,dxdt\\
&-\varepsilon^2\int_0^\tau\int_{\Omega} \nabla\beta(\rho)\otimes\nabla\beta(\rho):\nabla U + \frac12 K''(\rho)|\nabla\rho|^2\diver U + \nabla K(\rho)\cdot\nabla\diver U\,dxdt,
\end{aligned}
\end{equation}
where $\mathcal E$ is the standard relative energy introduced in \eqref{def:rel}.

Next, we recall the definition of the augmented relative energy \eqref{eq:aug_rel},
\begin{equation*}
\mathcal E_h(\tau)=\mathcal E(\tau)+\varepsilon^2\int_{\Omega}\left( \frac{1}{2}\rho|V|^2 - m\cdot V \right)dx.
\end{equation*}
To evaluate the time evolution of the augmented terms, we use the weak formulation of the continuity equation \eqref{weak1} tested against $-\frac{\varepsilon^2}{2}|V|^2$, and the weak formulation of the augmented momentum equation \eqref{eq: weak v} tested against $\varepsilon^2 V$. Notice that the spatial regularity $V \in C^2(\overline{\Omega})^d$ guarantees that $\nabla \diver V$ is well-defined and bounded, making $V$ a rigorously admissible test function. Summing these two contributions yields:
\begin{equation}\label{eq:general_weak_aug}
\begin{aligned}
\varepsilon^2\int_{\Omega}\left(\frac{1}{2}\rho|V|^2 - m\cdot V\right)dx \Bigg|_{t=0}^\tau =&
\varepsilon^2\int_0^\tau\int_{\Omega} (\rho V-m)\cdot\d_tV + J^k V^j \d_k V^j\,dxdt\\
&+\varepsilon^2\int_0^\tau\int_{\Omega} \mu'(\rho)J\cdot\nabla\diver V - \Lambda\cdot\big(\sqrt{\rho}\mu''(\rho)\nabla\rho\big)\diver V\,dxdt.
\end{aligned}
\end{equation}
By adding identity \eqref{eq:general_weak_aug} to inequality \eqref{eq:general_weak}, we recover the augmented inequality \eqref{eq:aug_rel_b}.
\end{proof}

\subsection{Boundary layer correction}\label{sec kato correction}

In order to show the convergence stated in Theorem \ref{main high}, we would like to use $(r, U, V)=(\rho^E, u^E, v(\rho^E))$ as test functions in Lemma \ref{lem:rel_hi}. However, in general, $\rho^E$ does not satisfy the homogeneous Neumann boundary condition \eqref{BC}, and hence $v(\rho^E)\cdot n|_{\d\Omega}\neq 0$. For this reason, we must introduce a boundary layer correction $v_{bl}$, so that the corrected velocity $v^E_{bl} := v(\rho^E)-v_{bl}$ satisfies $v^E_{bl}\cdot n|_{\d\Omega}=0$. 

To this end, let $d_\Omega(x) = \text{dist}(x, \partial\Omega)$ be the distance function to the boundary. We introduce a smooth cut-off function $\chi \in C^\infty([0, \infty))$ satisfying the following properties,
\begin{equation*}
    0 \le \chi \le 1, \quad \chi(0)=1, \quad \chi'(0)=0, \quad \text{and} \quad \chi(s)=0 \text{ for } s \ge 1.
\end{equation*}
Recalling the definition of $\theta$ given in \eqref{def:theta}, we then define the boundary layer corrector as
\begin{align}
    v_{bl}(\rho^E) &=\nabla\left[\chi\left(\frac{d_\Omega(x)}{c\delta}\right)\theta(\rho^E)\right] \label{v_bl}\\
    &=\chi\left(\frac{d_\Omega(x)}{c\delta}\right)v(\rho^E) + \theta(\rho^E)\chi'\left(\frac{d_\Omega(x)}{c\delta}\right)\frac{\nabla d_\Omega}{c\delta}, \notag
\end{align}
where $c>0$ is a constant and $\delta>0$ is the boundary layer thickness parameter.
Notice that, thanks to the condition $\chi'(0)=0$, on the boundary $\partial\Omega$ we have
\begin{equation}
    v_{bl} \cdot n = v(\rho^E)\cdot n.
\end{equation} 
Moreover, for $\rho^E$ as in Theorem \ref{theo sol euler}, since $\chi'$ is bounded and scaled by $\delta$, we have the following asymptotic bounds
\begin{align}
\label{asintotico}
    \|v_{bl}\|_{L^\infty} &= O\left(\delta^{-1}\right) \quad \text{in } C([0,T]; L^\infty(\Omega)), \notag\\
    \|\d_t v_{bl}\|_{L^\infty} &= O\left(\delta^{-1}\right) \quad \text{in } C([0,T]; L^\infty(\Omega)), \\
    \|D^\alpha_x v_{bl}\|_{L^\infty} &= O\left(\delta^{-|\alpha|-1}\right) \quad \text{in } C([0,T]; L^\infty(\Omega)), \notag
\end{align}
for any multi-index $\alpha\in\mathbb{N}^d$. 

For the sake of brevity, in what follows we denote the Euler quantum velocity as $v^E = v(\rho^E)$. We finally set our admissible test function to be
\begin{equation}\label{v boundary layer}
    v^E_{bl} = v^E - v_{bl}(\rho^E).
\end{equation}
By construction, this corrected profile rigorously satisfies the homogeneous Neumann boundary condition 
\begin{equation*}
    v^E_{bl} \cdot n|_{\partial\Omega} = 0.
\end{equation*}

\begin{lem}\label{lem:rel1 high}
    Let $(\rho,J)$ be a dissipative weak solution to \eqref{EK_w}-\eqref{BC_w} in the sense of Definition \ref{def: WS-EK}, let $( \rho^E , u^E ) \in C^1([0,T]; C^3(\overline{\Omega}) \times C^2(\overline{\Omega})^d)$ be the strong Euler solution bounded away from vacuum ($\rho^E > 0$), and let $v^E_{bl} \in C^1([0,T]; C^2(\overline{\Omega})^d)$ be defined as in \eqref{v boundary layer}. 
    
    Then, for a.e. $\tau\in (0,T)$, we have 
    \begin{align}\label{high relative entropy}
        \mathcal{E}_h(\tau) \leq \mathcal{E}_h(0) + \int_0^\tau \mathcal{R}_h(t)\,dt,
    \end{align}
where the total remainder is decomposed as
\begin{equation}\label{eq:three_rem}
    \mathcal{R}_h = \mathcal{R}_{rel} + \mathcal{R}_{bl} + \mathcal{R}_{in}.
\end{equation}
Here, $\mathcal{R}_{rel}$ is given by
\begin{align*}
    \mathcal{R}_{rel} =& -\int_\Omega \left[(\Lambda-\sqrt{\rho}u^E)\otimes(\Lambda-\sqrt{\rho}u^E) + \varepsilon^2(\nabla\beta(\rho)-\sqrt{\rho}v^E)\otimes(\nabla\beta(\rho)-\sqrt{\rho}v^E)\right] : \nabla u^E \\
    &+\varepsilon^2\int_\Omega (m-\rho v^E)\cdot\left[\nabla(\mu'(\rho^E)-\mu'(\rho))\diver u^E+ (\mu'(\rho^E)-\mu'(\rho))\nabla\diver u^E\right]\\
    &+\varepsilon^2\int_\Omega (J-\rho u^E)\cdot\left[\nabla(\mu'(\rho)-\mu'(\rho^E))\diver v^E+ (\mu'(\rho)-\mu'(\rho^E))\nabla\diver v^E\right]\\
    &-\int_\Omega \big[p(\rho)-p(\rho^E)-(\rho-\rho^E)p'(\rho^E)\big]\diver u^E,
\end{align*}
the boundary layer error $\mathcal{R}_{bl}$,
\begin{align*}
    \mathcal{R}_{bl} =& -\varepsilon^2\int_\Omega (\rho v^E_{bl}-m)\cdot \d_t v_{bl} + \varepsilon^2\int_\Omega J^j\left(v^{k}_{bl}\d_jv^{k}_{bl}-v^{k}_{bl}\d_jv^{E,k}-v^{E,k}\d_jv^{k}_{bl}\right)\\
    &+\varepsilon^2\int_\Omega\rho v_{bl}^j v^{E,k}\partial_j u^{E,k} + \varepsilon^2\int_\Omega \rho v_{bl}\cdot\nabla(\mu'(\rho^E)-\mu'(\rho))\diver u^E\\
    &+\varepsilon^2\int_\Omega \rho (\mu'(\rho^E)-\mu'(\rho) )v_{bl}\cdot\nabla\diver u^E -\varepsilon^2\int_\Omega \mu'(\rho)J\cdot \nabla\diver v_{bl}\\
    &-\varepsilon^2\int_\Omega \rho\mu'(\rho)\diver v_{bl}\diver u^E -\varepsilon^2\int_\Omega J\cdot \nabla \mu'(\rho)\diver v_{bl} \\
    &+ \varepsilon^2 \int_\Omega \Big( m^j u^{E,k}\d_k v^j_{bl} - m^j v^k_{bl}\d_k u^{E,j} - m^j u^{E,j}\d_k v^k_{bl} \Big) \, dx.
\end{align*}
and the remainder $\mathcal{R}_{in}$ is,
\begin{equation*}
    \mathcal{R}_{in} = -\varepsilon^2\int_\Omega (\rho u^E-J)\cdot\nabla\left(\mu'(\rho^E)\diver v^E+\frac{1}{2}|v^E|^2\right).
\end{equation*}
\end{lem}

\begin{rmk}
Let us clarify the roles of the three remainder terms appearing in \eqref{eq:three_rem}. $\mathcal{R}_{rel}$ is the higher-order counterpart of the relative energy \eqref{relative entropy}, containing squares for both the kinetic and the capillary relative velocities. The term $\mathcal{R}_{bl}$ accounts for the errors generated by the boundary layer corrector $v_{bl}$. Finally, the contribution $\mathcal{R}_{in}$ arises because the strong Euler solution $(\rho^E, u^E)$ does not exactly solve the full Korteweg system \eqref{EK}. %Let us also remark that all terms in $\mathcal{R}_{rel}$ are well-defined for dissipative weak solutions, in particular we recall \eqref{eq:integrab_condition}.
\end{rmk}
\begin{proof}
We apply the identity provided by Lemma \ref{lem:rel_hi}, choosing the test functions $(r, U, V)=(\rho^E, u^E, v^E_{bl})$. We exploit the  equations of the Euler profiles $u^E$ and $v^E$ derived in \eqref{eq u e v}. To isolate the limit remainder generated by the Euler system, we add and subtract the Korteweg tensor evaluated at $(\rho^E,u^E)$.
The negative part forms the remainder $\mathcal{R}_{in}$. Rearranging the remaining terms by lines, we obtain
\begin{equation}\label{eq:201}
\begin{aligned}
    \mathcal{E}_h(\tau) &- \mathcal{E}_h(0) \\&\leq \int_0^\tau \mathcal{R}_{in} dt + \int_0^\tau\int_\Omega (\rho u^E-J)\cdot \left[-u^E\cdot \nabla u^E-\nabla f'(\rho^E) + \varepsilon^2\nabla\Big(\mu'(\rho^E)\diver v^E+\frac{1}{2}|v^E|^2\Big)\right] \,dx \\
    &+\varepsilon^2\int_0^\tau\int_\Omega (\rho v^E_{bl}-m)\cdot\left[-\nabla(u^E\cdot v^E)-\nabla(\mu'(\rho^E)\diver u^E) -\partial_t v_{bl}\right]\,dx\\
    &+\int_0^\tau\int_\Omega (J^ju^{E,k}-\Lambda^j\Lambda^k)\d_ju^{E,k} - p(\rho)\diver u^E + (\rho-\rho^E)\partial_t f'(\rho^E) + J\cdot\nabla f'(\rho^E)\,dx\\
    &+\varepsilon^2\int_0^\tau\int_\Omega J^j v^{E,k}_{bl}\d_j v_{bl}^{E,k} + \mu'(\rho)J\cdot \nabla\diver v^E_{bl} - \Lambda\cdot \big(\sqrt{\rho}\mu''(\rho)\nabla\rho\big)\diver v^E_{bl}\,dx\\
    &-\varepsilon^2\int_0^\tau\int_\Omega \nabla\beta(\rho)\otimes\nabla\beta(\rho):\nabla u^E + m\cdot\nabla\mu'(\rho)\diver u^E + \mu'(\rho)m\cdot\nabla(\diver u^E)\,dx,
\end{aligned}
\end{equation}
where in the last line we used the identities 
\begin{equation*}
    \frac{1}{2}K''(\rho)|\nabla\rho|^2 = m\cdot\nabla\mu'(\rho) \quad \text{and} \quad \nabla K(\rho) = \mu'(\rho)m.
\end{equation*}
We now complete the perfect squares. For the kinetic relative velocity, we combine 
\begin{equation}\label{eq:proof_kin}
    \int_\Omega (J^ju^{E,k}-\Lambda^j\Lambda^k)\d_ju^{E,k} - (\rho u^E-J)\cdot (u^E\cdot \nabla) u^E \,dx = -\int_\Omega (\Lambda-\sqrt{\rho}u^E)\otimes(\Lambda-\sqrt{\rho}u^E) : \nabla u^E \,dx.
\end{equation}
Analogously, we reconstruct the capillary relative velocity square. Recalling the definition \eqref{v boundary layer} and expanding the gradients, we isolate the perfect square and its boundary layer remainders,
\begin{align}\label{eq:proof_cap}
    -\varepsilon^2\int_\Omega &(\rho v^E_{bl}-m)\cdot \nabla(u^E\cdot v^E) + \nabla\beta(\rho)\otimes\nabla\beta(\rho):\nabla u^E \,dx \notag\\
    =& -\varepsilon^2\int_\Omega (\nabla\beta(\rho)-\sqrt{\rho}v^E)\otimes(\nabla\beta(\rho)-\sqrt{\rho}v^E):\nabla u^E \,dx \notag\\
    &- \varepsilon^2\int_\Omega m^k v^{E,j}\d_j u^{E,k} \,dx - \varepsilon^2\int_\Omega (\rho v^E-m)^j u^{E,k}\d_j v^{E,k} \,dx \notag\\
    &+ \varepsilon^2\int_\Omega \rho v_{bl}^j (v^{E,k}\partial_j u^{E,k} + u^{E,k}\d_j v^{E,k}) \,dx.
\end{align}
By adding and subtracting $\rho v^{E,k} u^{E,j}\d_j v^{E,k}$, the cross-terms in the second line of \eqref{eq:proof_cap} become
\begin{align}\label{eq:proof_irrot}
    -\varepsilon^2\int_\Omega (\rho v^E-m)^j u^{E,k}\d_j v^{E,k} \,dx &= \varepsilon^2\int_\Omega m^j u^{E,k}\d_j v^{E,k} \,dx - \varepsilon^2\int_\Omega \rho v^{E,j} u^{E,k}\d_j v^{E,k} \,dx \notag\\
    &= \varepsilon^2\int_\Omega m^j u^{E,k}\d_j v^{E,k} \,dx -\varepsilon^2 \int_\Omega \rho v^{E,k} u^{E,j}\d_j v^{E,k} \,dx \notag\\
    &\quad +\varepsilon^2 \int_\Omega \rho (v^{E,k}u^{E,j} - v^{E,j}u^{E,k})\d_j v^{E,k} \,dx.
\end{align}
Since $v^E$ is irrotational, i.e. $v^E = \nabla\theta(\rho^E)$, we have $\partial_j v^{E,k} = \partial_k v^{E,j}$.
Thus, the last integral in \eqref{eq:proof_irrot} vanishes, resulting in 
\begin{align}\label{eq:proof_irrot2}
    -\varepsilon^2\int_\Omega (\rho v^E-m)^j u^{E,k}\d_j v^{E,k} \,dx 
    &= \varepsilon^2\int_\Omega m^j u^{E,k}\d_j v^{E,k} \,dx - \varepsilon^2\int_\Omega \rho v^{E,k} u^{E,j}\d_j v^{E,k} \,dx.
\end{align}
Notice that the added term in \eqref{eq:201} expands as follows
\begin{equation}
    \varepsilon^2 \int (\rho u^E - J) \cdot \nabla (\frac{1}{2}|v^E|^2)=\varepsilon^2 \int \rho u^{E,j} v^{E,k}\d_j v^{E,k} - \varepsilon^2 \int J^j v^{E,k}\d_j v^{E,k}.
\end{equation}
The first part is canceled by the second integral on the right hand side of \eqref{eq:proof_irrot2}. On the other hand, the second part is canceled by the first term of the following expansion from \eqref{eq:201}: 
\begin{equation*}
    \varepsilon^2\int_\Omega J^j v^{E,k}_{bl}\d_j v_{bl}^{E,k}\, dx=\varepsilon^2\int_\Omega J^j\left(v^{E,k}\d_j v^{E,k}+v^{k}_{bl}\d_jv^{k}_{bl}-v^{k}_{bl}\d_jv^{E,k}-v^{E,k}\d_jv^{k}_{bl}\right)\, dx.
\end{equation*}
Next, we handle the high-order Korteweg stress terms. We aim to pair the gradients of the divergence $\nabla \diver u^E$ and $\nabla \diver v^E$. Expanding $v^E_{bl}$, 
\begin{multline}\label{eq:proof_mob1}
    m\cdot\nabla \big(\mu'(\rho^E)-\mu'(\rho)\big)\diver u^E -\rho v^E_{bl}\cdot\nabla \mu'(\rho^E)\diver u^E +J\cdot\nabla\big(\mu'(\rho)-\mu'(\rho^E)\big)\diver v^E+\rho u^E\cdot \nabla\mu'(\rho^E)\diver v^E \\
    = (m-\rho v^E_{bl})\cdot\nabla\big(\mu'(\rho^E)-\mu'(\rho)\big)\diver u^E + (J-\rho u^E)\cdot\nabla\big(\mu'(\rho)-\mu'(\rho^E)\big)\diver v^E \\
    -\rho v^E_{bl}\cdot \nabla\mu'(\rho)\diver u^E + \rho u^E\cdot \nabla \mu'(\rho)\diver v^E.
\end{multline}
We integrate by parts the last two terms of \eqref{eq:proof_mob1}. Using the homogeneous boundary condition $v^E_{bl}\cdot n|_{\d \Omega}=0$, we obtain
\begin{align}\label{eq:proof_mob2}
    -\int_\Omega \rho v^E_{bl}\cdot\nabla\mu'(\rho)\diver u^E \,dx &= \int_\Omega\rho\mu'(\rho)\diver v^E_{bl}\diver u^E \,dx+ \mu'(\rho)v^E_{bl}\cdot\nabla\rho\diver u^E \,dx\notag\\&\quad+ \int_\Omega\rho\mu'(\rho)v^E_{bl}\cdot \nabla\diver u^E \,dx.
\end{align}
Analogously, for the velocity $u^E$,
\begin{align}\label{eq:proof_mob3}
    \int_\Omega \rho u^E\cdot\nabla\mu'(\rho)\diver v^E \,dx =& -\int_\Omega\rho\mu'(\rho)\diver v^E\diver u^E + \mu'(\rho)u^E\cdot\nabla\rho\diver v^E\,dx\notag
    \\&- \int_\Omega\rho\mu'(\rho)u^E\cdot \nabla\diver v^E \,dx.
\end{align}
Pairing \eqref{eq:proof_mob2} and \eqref{eq:proof_mob3} with $\nabla \diver u^E$ and $\nabla \diver v^E$ terms in \eqref{eq:201}, we get
\begin{align*}
    \nabla\diver u^E \cdot \Big(m\mu'(\rho^E)-m\mu'(\rho)-\rho \mu'(\rho^E)v^E_{bl}+\rho\mu'(\rho)v^E_{bl}\Big) &= \nabla\diver u^E \cdot (m-\rho v^E_{bl})\big(\mu'(\rho^E)-\mu'(\rho)\big),\\
    \nabla\diver v^E \cdot \Big(J\mu'(\rho)-J\mu'(\rho^E)+\rho \mu'(\rho^E)u^E-\rho\mu'(\rho)u^E\Big) &= \nabla\diver v^E \cdot (J-\rho u^E)\big(\mu'(\rho)-\mu'(\rho^E)\big).
\end{align*}
Finally, we isolate the remaining cross-terms involving the momentum $m$. Recalling that $\nabla \mu(\rho)=m$, from \eqref{eq:proof_cap}, \eqref{eq:proof_irrot2}, \eqref{eq:proof_mob2} and \eqref{eq:proof_mob3}, we have
\begin{align}\label{eq:proof_m_collect}
    &\varepsilon^2 \int_\Omega \Big( m^j u^{E,k}\d_j v^{E,k} - m^k v^{E,j}\d_j u^{E,k} + m^j v^{E,j}_{bl} \d_k u^{E,k} - m^j u^{E,j}\d_k v^{E,k} \Big) \, dx
\end{align}
Using again the irrotationality of $v^E$, the first term rewrites as $m^j u^{E,k}\d_k v^{E,j}$. To avoid boundary terms when integrating by parts, we substitute the decomposition $v^E = v^E_{bl} + v_{bl}$ into the expression \eqref{eq:proof_m_collect}. This allows us to group all boundary-layer-corrected velocities:
\begin{align}\label{eq:proof_m_split}
     &\varepsilon^2 \int_\Omega \Big( m^j u^{E,k}\d_k v^{E,j}_{bl} - m^j v^{E,k}_{bl}\d_k u^{E,j} + m^j v^{E,j}_{bl} \d_k u^{E,k} - m^j u^{E,j}\d_k v^{E,k}_{bl} \Big) \, dx \notag\\
    + &\varepsilon^2 \int_\Omega \Big( m^j u^{E,k}\d_k v^j_{bl} - m^j v^k_{bl}\d_k u^{E,j} - m^j u^{E,j}\d_k v^k_{bl} \Big) \, dx.
\end{align}
We now show that the first integral evaluated at $v^E_{bl}$ is identically zero. We integrate by parts the last two terms of the first line. Since $u^E \cdot n = 0$ and $v^E_{bl} \cdot n = 0$ on $\partial \Omega$, we derive
\begin{align*}
    \int_\Omega m^j v^{E,j}_{bl} \d_k u^{E,k} &- m^j u^{E,j}\d_k v^{E,k}_{bl} \, dx\\
    &= \int_\Omega -\d_k(m^j v^{E,j}_{bl})u^{E,k} + \d_k(m^j u^{E,j})v^{E,k}_{bl} \, dx \\
    &= \int_\Omega \big( - \d_k m^j v^{E,j}_{bl}u^{E,k} - m^j \d_k v^{E,j}_{bl}u^{E,k} + \d_k m^j u^{E,j}v^{E,k}_{bl} + m^j \d_k u^{E,j} v^{E,k}_{bl} \big) \,dx.
\end{align*}
Since the capillary momentum $m$ is irrotational, we have $\d_k m^j = \d_j m^k$. Thus, the two terms involving derivatives of $m$ cancel against each other. The remaining terms cancel the first two terms in the first integral block of \eqref{eq:proof_m_split}. 

Collecting the identities \eqref{eq:proof_kin}--\eqref{eq:proof_m_split} yields the exact decomposition \eqref{eq:three_rem} concluding the proof.
\end{proof}

\subsection{Estimates on $\mathcal{R}_h$}\label{sec estimate R high}

To simplify the presentation in the upcoming bounds, we introduce a notation for the augmented relative energy evaluated at the strong Euler solution, both with and without the boundary layer correction
\begin{notat}
    Whenever it does not create ambiguity, we write
    \begin{equation*}
        \mathcal{E}_h^E(\tau) := \mathcal{E}_h(\tau) \Big|_{(r,U,V) = (\rho^E, u^E, v^E)}.
    \end{equation*}
    On the other hand, we simply write $\mathcal{E}_h(\tau)$ to denote the augmented relative energy evaluated at the corrected functions $(\rho^E, u^E, v^E_{bl})$.
\end{notat}
\begin{prop}\label{prop high R}
Let $\delta = \varepsilon^s$ be the boundary layer thickness. For a.e. $\tau \in (0,T)$, there exists a constant $C>0$ (independent of $\varepsilon$) such that the remainder terms defined in \eqref{eq:three_rem} satisfy the following bounds,
\begin{align}
    |\mathcal{R}_{rel}| &\leq C \mathcal{E}_h^E(\tau) + O(\varepsilon^p), \label{est:R_rel} \\
    |\mathcal{R}_{bl}| &\leq C \mathcal{E}_h^E(\tau) + O(\varepsilon^s), \label{est:R_bl} \\
    |\mathcal{R}_{in}| &\leq C \mathcal{E}_h^E(\tau) + O(\varepsilon^4), \label{est:R_in}
\end{align}
where $p,s>0$ satisfy the following constraints,
\begin{itemize}
    \item For $d=1$, \quad $s < \min\left\{\frac{1}{2}, \frac{5+\alpha}{3(3+\alpha)}\right\}$,\quad \quad \quad $p=\frac{4}{3+\alpha}$,
    \item For $d=2$, \quad $s < \min\left\{\frac{1}{2}, \frac{3+\alpha}{3(2+\alpha)}\right\}$,\quad \quad \quad $p=\frac{2}{2+\alpha}$,
    \item For $d\geq 3$, \quad $s < \min\left\{\frac{1}{2}, \frac{d(1+\alpha)+4}{3(d(1+\alpha)+2)}\right\}$,\quad $p=\frac{4}{d(1+\alpha)+2}$.
\end{itemize}
\end{prop}

\begin{proof}
Let us start with the limit residual $\mathcal{R}_{in}$. Since the Euler solution $(\rho^E, u^E, v^E)$ is sufficiently smooth and bounded away from vacuum, the $L^\infty$-norm of the gradient is bounded by a constant $C$. Moreover, by conservation of mass, $\|\sqrt{\rho}\|_{L^2}^2 = \int_\Omega \rho \, dx \leq C$. Recalling that $J = \rho u$, an application of the Cauchy-Schwarz inequality yields:
\begin{align*}
    |\mathcal{R}_{in}| &\leq \varepsilon^2 \norm{\nabla\left(\mu'(\rho^E)\diver v^E+\frac{1}{2}|v^E|^2\right)}_{L^\infty} \norm{\sqrt{\rho}}_{L^2} \norm{\sqrt{\rho}(u^E-u)}_{L^2}\leq C\varepsilon^2 (\mathcal{E}^E_h)^{1/2}.
\end{align*}
Applying Young's inequality to the right-hand side, we conclude,
\begin{equation}\label{est:proof_R_in}
    |\mathcal{R}_{in}| \leq C\mathcal{E}^E_h + O(\varepsilon^4).
\end{equation}
Let us now focus on $\mathcal{R}_{rel}$. For the first line, dealing with the kinetic and capillary relative velocities, we exploit the uniform bound on the Euler velocity gradient to obtain,
\begin{align*}
    \Bigg| \int_\Omega \Big[ (\Lambda-\sqrt{\rho}u^E)\otimes(\Lambda-\sqrt{\rho}u^E) &+ \varepsilon^2(\nabla\beta(\rho)-\sqrt{\rho}v^E)\otimes(\nabla\beta(\rho)-\sqrt{\rho}v^E) \Big] : \nabla u^E \,dx \Bigg| \\
    &\leq \|\nabla u^E\|_{L^\infty} \int_\Omega \Big( |\Lambda-\sqrt{\rho}u^E|^2 + \varepsilon^2|\nabla\beta(\rho)-\sqrt{\rho}v^E|^2 \Big) \,dx \\
    &\leq C \|\nabla u^E\|_{L^\infty} \mathcal{E}_h^E(\tau) \\
    &\leq C \mathcal{E}_h^E(\tau).
\end{align*}
The next term in $\mathcal{R}_{rel}$ is,
\begin{equation}\label{eq:R_rel_mob}
    \varepsilon^2\int_\Omega (m-\rho v^E)\cdot\Big[\nabla\big(\mu'(\rho^E)-\mu'(\rho)\big)\diver u^E+ \big(\mu'(\rho^E)-\mu'(\rho)\big)\nabla\diver u^E\Big] \,dx.
\end{equation}
For the first term inside the bracket, we note that by definition,
\begin{equation}\label{eq_nabla_mu'}
    \nabla \mu'(\rho) = \frac{\alpha+1}{2} \frac{1}{\sqrt{\rho}}\nabla\beta(\rho).
\end{equation}
 Consequently, pairing the difference of the gradients with the relative momentum $m-\rho v^E = \sqrt{\rho}(\nabla\beta(\rho) - \sqrt{\rho}v^E)$ reconstructs the capillary relative velocity square. Since $\|\diver u^E\|_{L^\infty} \leq C$, we obtain
\begin{align}
    \left| \varepsilon^2\int_\Omega (m-\rho v^E)\cdot\nabla\big(\mu'(\rho^E)-\mu'(\rho)\big)\diver u^E \,dx \right| &\leq C \|\diver u^E\|_{L^\infty} \varepsilon^2 \norm{\nabla\beta(\rho)-\sqrt{\rho}v^E}^2_{L^2} \notag\\
    &\leq C \mathcal{E}^E_h(\tau).\label{est:nabla_mu_term}
\end{align}
For the second term in \eqref{eq:R_rel_mob}, we use the following decomposition,
\begin{align}\label{eq:R_rel_mob2}
    \varepsilon^2\int_\Omega (m-\rho v^E)\cdot \big(\mu'(\rho^E)-\mu'(\rho)\big)\nabla\diver u^E \,dx &= \varepsilon^2\int_\Omega (m-\rho v^E)\cdot \mu'(\rho^E)\nabla\diver u^E \,dx\\
    &\quad - \varepsilon^2\int_\Omega (m-\rho v^E)\cdot \mu'(\rho)\nabla\diver u^E \,dx\notag\\
    &=: A + B.\notag
\end{align}
For the smooth part $A$, bounding the smooth terms  in $L^\infty$, we get
\begin{align*}
    |A| &\leq \varepsilon^2 \norm{\sqrt{\rho}}_{L^2} \norm{\nabla \beta(\rho)-\sqrt{\rho}v^E}_{L^2} \norm{\mu'(\rho^E)\nabla\diver u^E}_{L^\infty} \leq C \varepsilon (\mathcal{E}^E_h)^{1/2} \leq C \mathcal{E}^E_h(\tau) + O(\varepsilon^2).
\end{align*}
On the other hand, for the non-linear part $B$, we absorb the density to form $\sqrt{\rho}\mu'(\rho) \sim \rho^{(\alpha+2)/2}$:
\begin{align}\label{eq:B}
    |B| &\leq \varepsilon^2 \norm{\nabla \beta(\rho)-\sqrt{\rho}v^E}_{L^2} \norm{\rho}_{L^{\alpha+2}}^{\frac{\alpha+2}{2}} \norm{\nabla\diver u^E}_{L^\infty} \leq C \varepsilon (\mathcal{E}^E_h)^{1/2} \norm{\rho}_{L^{\alpha+2}}^{\frac{\alpha+2}{2}}.
\end{align}
By applying Lemma \ref{lemma: GN inequality} to bound the $L^{\alpha+2}$ norm via the capillary energy, and subsequently using Young's inequality, we obtain 
\begin{equation*}
    |B| \leq 
    \begin{cases}
        C \varepsilon^{\frac{2}{3+\alpha}}(\mathcal{E}^E_h)^{\frac{1}{2}} \leq C\mathcal{E}^E_h(\tau) + O\left(\varepsilon^{\frac{4}{3+\alpha}}\right) & \text{if } d=1, \\[6pt]
        C \varepsilon^{\frac{1}{2+\alpha}}(\mathcal{E}^E_h)^{\frac{1}{2}} \leq C\mathcal{E}^E_h(\tau) + O\left(\varepsilon^{\frac{2}{2+\alpha}}\right) & \text{if } d=2, \\[6pt]
        C \varepsilon^{\frac{2}{d(1+\alpha)+2}}(\mathcal{E}^E_h)^{\frac{1}{2}} \leq C\mathcal{E}^E_h(\tau) + O\left(\varepsilon^{\frac{4}{d(1+\alpha)+2}}\right) & \text{if } d \ge 3.
    \end{cases}
\end{equation*}
Similarly, we study the momentum cross-term inside $\mathcal{R}_{rel}$,
\begin{equation}\label{eq:R_rel_J_mob}
    \varepsilon^2\int_\Omega (J-\rho u^E)\cdot\left[\nabla\big(\mu'(\rho)-\mu'(\rho^E)\big)\diver v^E+ \big(\mu'(\rho)-\mu'(\rho^E)\big)\nabla\diver v^E\right] \,dx.
\end{equation}
For the first term, we use again \eqref{eq_nabla_mu'} and the relation $J-\rho u^E = \sqrt{\rho}(\Lambda-\sqrt{\rho}u^E)$. Since $v^E$ is the strong Euler solution, its divergence is uniformly bounded, yielding
\begin{align}\label{est:J_nabla_mu}
    \Bigg| \varepsilon^2\int_\Omega (J-\rho u^E)\cdot \nabla\big(\mu'(\rho)-\mu'(\rho^E)\big)\diver v^E \,dx \Bigg| &\leq C \varepsilon^2 \|\diver v^E\|_{L^\infty} \norm{\nabla\beta(\rho)-\sqrt{\rho}v^E}_{L^2}\norm{\Lambda-\sqrt{\rho}u^E}_{L^2} \notag\\
    &\leq C \varepsilon (\mathcal{E}_h^E)^{1/2}(\mathcal{E}_h^E)^{1/2} \notag\\
    &\leq C \mathcal{E}_h^E(\tau).
\end{align}
On the other hand, the second term in \eqref{eq:R_rel_J_mob} requires the same splitting  used in \eqref{eq:R_rel_mob2}. We bound the regular and the non-linear contributions separately,
\begin{align*}
    \bigg| \varepsilon^2\int_\Omega (J-\rho u^E) \cdot \big(\mu'(\rho)&-\mu'(\rho^E)\big) \nabla\diver v^E \,dx \bigg| \\
    &\leq \varepsilon^2 \norm{\Lambda-\sqrt{\rho}u^E}_{L^2} \Big( \norm{\mu'(\rho^E)}_{L^\infty}\norm{\sqrt{\rho}}_{L^2} + \norm{\sqrt{\rho}\mu'(\rho)}_{L^2} \Big) \norm{\nabla\diver v^E}_{L^\infty} \\
    &\leq C \varepsilon^2 (\mathcal{E}^E_h)^{1/2} \Big( 1 + \norm{\rho}_{L^{\alpha+2}}^{\frac{\alpha+2}{2}} \Big) \\
    &\leq C \mathcal{E}^E_h(\tau) + O(\varepsilon^4) + C \varepsilon^2 (\mathcal{E}^E_h)^{\frac{1}{2}} \norm{\rho}_{L^{\alpha+2}}^{\frac{\alpha+2}{2}}.
\end{align*}
Using again Lemma \ref{lemma: GN inequality} together with Young's inequality, we establish the bounds for the non-linear term. Due to the presence of $\varepsilon^2$ (unlike the capillary term bounded in \eqref{eq:B}), we obtain 
\begin{equation*}
    C \varepsilon^2(\mathcal{E}^E_h)^{\frac{1}{2}}\norm{\rho}_{L^{\alpha+2}}^{\frac{\alpha+2}{2}} \leq
    \begin{cases}
        C \varepsilon^{\frac{5+\alpha}{3+\alpha}}(\mathcal{E}^E_h)^{\frac{1}{2}} \leq C\mathcal{E}^E_h(\tau) + O\left(\varepsilon^{\frac{2(5+\alpha)}{3+\alpha}}\right) & \text{if } d=1, \\[6pt]
        C \varepsilon^{\frac{3+\alpha}{2+\alpha}}(\mathcal{E}^E_h)^{\frac{1}{2}} \leq C\mathcal{E}^E_h(\tau) + O\left(\varepsilon^{\frac{2(3+\alpha)}{2+\alpha}}\right) & \text{if } d=2, \\[6pt]
        C \varepsilon^{\frac{d(1+\alpha)+4}{d(1+\alpha)+2}}(\mathcal{E}^E_h)^{\frac{1}{2}} \leq C\mathcal{E}^E_h(\tau) + O\left(\varepsilon^{\frac{2(d(1+\alpha)+4)}{d(1+\alpha)+2}}\right) & \text{if } d \ge 3.
    \end{cases}
\end{equation*}
Finally, we estimate the relative pressure term.  Exploiting the identity for the relative pressure, $p(\rho|\rho^E) = (\gamma-1)f(\rho|\rho^E)$, we obtain
\begin{align}\label{est:R_rel_pressure}
    \left|\int_\Omega \big[p(\rho)-p(\rho^E)-(\rho-\rho^E)p'(\rho^E)\big]\diver u^E \,dx \right| &= \left|\int_\Omega (\gamma-1)f(\rho|\rho^E)\diver u^E \,dx \right| \notag\\
    &\leq (\gamma-1) \norm{\diver u^E}_{L^\infty} \int_\Omega f(\rho|\rho^E) \,dx \notag\\
    &\leq C \mathcal{E}_h^E(\tau).
\end{align}
To control the relative remainder $\mathcal{R}_{rel}$, we take the minimum over all derived exponents of $O(\varepsilon)$. Thus, the convergence rate is given by
\begin{equation}\label{eq:p_rate_def}
    p := 
    \begin{cases}
        \frac{4}{3+\alpha} & \text{if } d=1, \\[6pt]
        \frac{2}{2+\alpha} & \text{if } d=2, \\[6pt]
        \frac{4}{d(1+\alpha)+2} & \text{if } d \ge 3.
    \end{cases}
\end{equation}
Therefore, the internal remainder satisfies
\begin{equation}\label{eq:R_rel_final_bound}
    |\mathcal{R}_{rel}| \leq C \mathcal{E}_h^E(\tau) + O(\varepsilon^p),
\end{equation}

We now proceed to bound the boundary layer remainder $\mathcal{R}_{bl}$. We begin with the first term involving the time derivative of the corrector. Recalling that  $v^E_{bl} = v^E - v_{bl}$, we split the integral into two terms as follows 
\begin{align}\label{est:R_bl_time_split}
    -\varepsilon^2\int_\Omega (\rho v^E_{bl}-m)\cdot \d_t v_{bl} \,dx &= -\varepsilon^2\int_\Omega (\rho v^E-m)\cdot \d_t v_{bl} \,dx + \varepsilon^2\int_\Omega \rho v_{bl} \cdot \d_t v_{bl} \,dx \notag\\
    &=: \mathcal{R}_{bl}^{1.1} + \mathcal{R}_{bl}^{1.2}.
\end{align}
For the first integral, we use the identity $\rho v^E - m = -\sqrt{\rho}(\nabla\beta(\rho) - \sqrt{\rho}v^E)$. Using the asymptotic bound on the time derivative \eqref{asintotico},  we obtain,
\begin{align*}
    |\mathcal{R}_{bl}^{1.1}| &\leq \varepsilon^2 \norm{\d_t v_{bl}}_{L^\infty} \norm{\sqrt{\rho}}_{L^2} \norm{\nabla\beta(\rho)-\sqrt{\rho}v^E}_{L^2} \leq C \varepsilon^2 \delta^{-1} \left(\varepsilon^{-1}(\mathcal{E}^E_h)^{1/2}\right) \leq  C \varepsilon^{1-s} (\mathcal{E}^E_h)^{1/2}.
\end{align*}
By applying Young's inequality, we absorb the energy term and we get the quadratic asymptotic error,
\begin{equation}\label{est:R_bl_time_1}
    |\mathcal{R}_{bl}^{1.1}| \leq C\mathcal{E}^E_h(\tau) + O(\varepsilon^{2-2s}).
\end{equation}
For the second integral $\mathcal{R}_{bl}^{1.2}$, there are no relative quantities. Using again the asymptotic bound \eqref{asintotico}, we get 
\begin{align}\label{est:R_bl_time_2}
    |\mathcal{R}_{bl}^{1.2}| &\leq \varepsilon^2 \norm{v_{bl}}_{L^\infty} \norm{\d_t v_{bl}}_{L^\infty} \norm{\rho}_{L^1} \leq C \varepsilon^2 \delta^{-2}  = O(\varepsilon^{2-2s}).
\end{align}
Next, we address the boundary layer terms involving $J$. This block consists of three cross-terms,
\begin{equation}\label{est:R_bl_J_split}
    \varepsilon^2\int_\Omega J^j\left(v^{k}_{bl}\d_jv^{k}_{bl} - v^{k}_{bl}\d_jv^{E,k} - v^{E,k}\d_jv^{k}_{bl}\right) \,dx.
\end{equation}
Since the Euler solution is smooth ($\|v^E\|_{L^\infty} + \|\nabla v^E\|_{L^\infty} \leq C$), the most singular contribution arises from the first term. Recalling \eqref{energy inequality}, we have $\|\Lambda\|_{L^2} \leq C$. Thus, an application of the Cauchy-Schwarz inequality yields,
\begin{align}\label{est:R_bl_conv_dom}
    \left| \varepsilon^2\int_\Omega \sqrt{\rho}\Lambda^j v^{k}_{bl}\d_jv^{k}_{bl} \,dx \right| &\leq \varepsilon^2 \norm{\sqrt{\rho}}_{L^2} \norm{\Lambda}_{L^2} \norm{v_{bl}}_{L^\infty} \norm{\nabla v_{bl}}_{L^\infty} \leq C \varepsilon^2 \delta^{-3}  = O(\varepsilon^{2-3s}).
\end{align}
It is straightforward to check that the remaining two cross-terms in \eqref{est:R_bl_J_split} scale respectively as $O(\varepsilon^{2-s})$ and $ O(\varepsilon^{2-2s})$. Since the boundary layer thickness $\delta = \varepsilon^s \ll 1$, these lower-order singularities are absorbed by the dominant  term \eqref{est:R_bl_conv_dom},
\begin{equation}\label{est:R_bl_conv_final}
    \left| \varepsilon^2\int_\Omega J^j\left(v^{k}_{bl}\d_jv^{k}_{bl} - v^{k}_{bl}\d_jv^{E,k} - v^{E,k}\d_jv^{k}_{bl}\right) \,dx \right| = O(\varepsilon^{2-3s}).
\end{equation}
Similarly, for the next term we use the uniform $L^\infty$ bounds on $v^E$ and $\nabla u^E$ and we get
\begin{align}\label{est:R_bl_smooth_cross}
    \left| \varepsilon^2\int_\Omega \rho v_{bl}^j v^{E,k}\partial_j u^{E,k} \,dx \right| &\leq \varepsilon^2 \norm{v^E}_{L^\infty} \norm{\nabla u^E}_{L^\infty} \norm{v_{bl}}_{L^\infty} \norm{\rho}_{L^1} \leq C \varepsilon^2 \delta^{-1}= O(\varepsilon^{2-s}).
\end{align}
Next, we examine the cross-term
\begin{equation}\label{est:R_bl_mob_grad}
    \varepsilon^2\int_\Omega \rho v_{bl}\cdot\nabla\big(\mu'(\rho^E)-\mu'(\rho)\big)\diver u^E \,dx.
\end{equation}
Recalling \eqref{eq_nabla_mu'}, the above integral is bounded as follows,
\begin{align}
    \Bigg| \varepsilon^2\int_\Omega \sqrt{\rho} v_{bl} \cdot \bigg( \frac{\alpha+1}{2} (\nabla\beta(\rho) - \sqrt{\rho}v^E) \bigg) \diver &u^E \,dx \Bigg|\notag\\ &\leq C \varepsilon^2 \norm{\diver u^E}_{L^\infty} \norm{v_{bl}}_{L^\infty} \norm{\sqrt{\rho}}_{L^2} \norm{\nabla\beta(\rho)-\sqrt{\rho}v^E}_{L^2} \notag\\
    &\leq C \varepsilon^2  \delta^{-1} \left( \varepsilon^{-1}(\mathcal{E}^E_h)^{1/2} \right) \label{est:R_bl_mob_final}\\
    &\leq C \varepsilon^{1-s} (\mathcal{E}^E_h)^{1/2}\notag\\
    &\leq C\mathcal{E}^E_h(\tau) + O(\varepsilon^{2-2s})\notag.
\end{align}
We decompose the fifth remainder term of $\mathcal{R}_{bl}$ into two contribution,
\begin{align}\label{est:R_bl_5_split}
     \varepsilon^2\int_\Omega \rho \mu'(\rho^E)v_{bl}\cdot\nabla\diver u^E \,dx - \varepsilon^2\int_\Omega \rho \mu'(\rho) v_{bl}\cdot\nabla\diver u^E \,dx =: \mathcal{R}_{bl}^{5.1} + \mathcal{R}_{bl}^{5.2}.
\end{align}
The first term $\mathcal{R}_{bl}^{5.1}$ is straightforward. Since $\mu'(\rho^E)$ is uniformly bounded, we get,
\begin{equation}\label{est:R_bl_5_1}
    |\mathcal{R}_{bl}^{5.1}| \leq \varepsilon^2 \norm{\mu'(\rho^E)}_{L^\infty} \norm{\nabla\diver u^E}_{L^\infty} \norm{v_{bl}}_{L^\infty} \norm{\rho}_{L^1} \leq C \varepsilon^2 \delta^{-1} = O(\varepsilon^{2-s}).
\end{equation}
For the second term $\mathcal{R}_{bl}^{5.2}$, we have $\rho\mu'(\rho) \sim \rho^{(\alpha+3)/2}$. We split the exponent to use conservation of mass and to apply Lemma \ref{lemma: GN inequality},
\begin{align}\label{est:R_bl_5_2_}
    |\mathcal{R}_{bl}^{5.2}| &\leq \varepsilon^2 \norm{\nabla \diver u^E}_{L^\infty}\norm{v_{bl}}_{L^\infty} \int_\Omega \rho^{(\alpha+3)/2} \,dx \notag\\
    &\leq C \varepsilon^2 \delta^{-1} \int_\Omega \sqrt{\rho} \rho^{\frac{\alpha+2}{2}} \,dx \notag\\
    &\leq C \varepsilon^{2-s} \norm{\sqrt{\rho}}_{L^2}\norm{\rho}_{L^{\alpha+2}}^{\frac{\alpha+2}{2}}.
\end{align}
Using Lemma \ref{lemma: GN inequality}, we obtain
\begin{equation}\label{est:R_bl_5_2_GN}
    |\mathcal{R}_{bl}^{5.2}| \leq 
    \begin{cases}
        C \varepsilon^{1+\frac{2}{3+\alpha}-s} & \text{if } d=1, \\[6pt]
        C \varepsilon^{1+\frac{1}{2+\alpha}-s} & \text{if } d=2, \\[6pt]
        C \varepsilon^{1+\frac{2}{d(1+\alpha)+2}-s} & \text{if } d \ge 3.
    \end{cases}
\end{equation}
Next, we analyze the most singular boundary layer interactions, given by the presence of two derivatives acting on the corrector. Recalling that $\mu'(\rho) = c_\alpha \rho^{(\alpha+1)/2}$ and the identity $J = \sqrt{\rho}\Lambda$, we rewrite the integral explicitly,
\begin{equation}\label{est:R_bl_nabla_diver_def}
    -\varepsilon^2\int_\Omega \mu'(\rho)J\cdot \nabla\diver v_{bl} \,dx = - c_\alpha \varepsilon^2\int_\Omega \rho^{\frac{\alpha+2}{2}} \Lambda\cdot \nabla\diver v_{bl} \,dx.
\end{equation}
Recalling that $\|\Lambda\|_{L^2} \leq C$, and noting from \eqref{asintotico} that $\|\nabla\diver v_{bl}\|_{L^\infty} = O(\delta^{-3}) = O(\varepsilon^{-3s})$, we obtain
\begin{align}\label{est:R_bl_nabla_diver_start}
    \left| \varepsilon^2\int_\Omega \mu'(\rho)J\cdot \nabla\diver v_{bl} \,dx \right| &\leq C \varepsilon^2 \norm{\rho}_{L^{\alpha+2}}^{\frac{\alpha+2}{2}} \norm{\Lambda}_{L^2} \norm{\nabla\diver v_{bl}}_{L^\infty} \notag\\
    &\leq C \varepsilon^2 \delta^{-3}\norm{\rho}_{L^{\alpha+2}}^{\frac{\alpha+2}{2}}   \notag\\
    &= C \varepsilon^{2-3s} \norm{\rho}_{L^{\alpha+2}}^{\frac{\alpha+2}{2}}.
\end{align}
We use again Lemma \ref{lemma: GN inequality} and we derive
\begin{equation}\label{est:R_bl_nabla_diver_GN}
    \left| \varepsilon^2\int_\Omega \mu'(\rho)J\cdot \nabla\diver v_{bl} \,dx \right| \leq 
    \begin{cases}
        C \varepsilon^{1+\frac{2}{3+\alpha}-3s} & \text{if } d=1, \\[6pt]
        C \varepsilon^{1+\frac{1}{2+\alpha}-3s} & \text{if } d=2, \\[6pt]
        C \varepsilon^{1+\frac{2}{d(1+\alpha)+2}-3s} & \text{if } d \ge 3.
    \end{cases}
\end{equation}
Similarly, we consider the next term
\begin{equation}\label{est:R_bl_diver_v_def}
    -\varepsilon^2\int_\Omega \rho\mu'(\rho)\diver v_{bl}\diver u^E \,dx.
\end{equation}
Proceeding exactly as in the estimate \eqref{est:R_bl_5_2_}, we write $\rho\mu'(\rho) \sim \sqrt{\rho}\rho^{(\alpha+2)/2}$ to exploit the conservation of mass. Since $\|\diver u^E\|_{L^\infty}$ is uniformly bounded and $\|\diver v_{bl}\|_{L^\infty} = O(\delta^{-2}) = O(\varepsilon^{-2s})$, we apply the Cauchy-Schwarz inequality to obtain,
\begin{align}\label{est:R_bl_diver_v_start}
    \left| \varepsilon^2\int_\Omega \rho\mu'(\rho)\diver v_{bl}\diver u^E \,dx \right| &\leq C \varepsilon^2 \norm{\diver u^E}_{L^\infty} \norm{\diver v_{bl}}_{L^\infty} \int_\Omega \sqrt{\rho}\rho^{\frac{\alpha+2}{2}} \,dx \notag\\
    &\leq C \varepsilon^2 \delta^{-2} \norm{\sqrt{\rho}}_{L^2} \norm{\rho}_{L^{\alpha+2}}^{\frac{\alpha+2}{2}} \notag\\
    &= C \varepsilon^{2-2s} \norm{\rho}_{L^{\alpha+2}}^{\frac{\alpha+2}{2}}.
\end{align}
Using Lemma \ref{lemma: GN inequality}, we deduce the corresponding asymptotic rates,
\begin{equation}\label{est:R_bl_diver_v_GN}
    \left| \varepsilon^2\int_\Omega \rho\mu'(\rho)\diver v_{bl}\diver u^E \,dx \right| \leq 
    \begin{cases}
        C \varepsilon^{1+\frac{2}{3+\alpha}-2s} & \text{if } d=1, \\[6pt]
        C \varepsilon^{1+\frac{1}{2+\alpha}-2s} & \text{if } d=2, \\[6pt]
        C \varepsilon^{1+\frac{2}{d(1+\alpha)+2}-2s} & \text{if } d \ge 3.
    \end{cases}
\end{equation}
Next, recalling \eqref{eq_nabla_mu'} and the identity $J = \sqrt{\rho}\Lambda$, we address  
\begin{equation*}
    -\varepsilon^2\int_\Omega J\cdot \nabla \mu'(\rho)\diver v_{bl} \,dx = -\frac{\alpha+1}{2}\varepsilon^2 \int_\Omega \Lambda \cdot \nabla\beta(\rho) \diver v_{bl} \,dx.
\end{equation*}
Notice that the energy inequality \eqref{energy inequality} yields the uniform bound $\|\Lambda\|_{L^2} \leq C$ and ensures that $\varepsilon^2\|\nabla\beta(\rho)\|_{L^2}^2 \leq C$. Furthermore, a spatial derivative acting on the corrector yields $\|\diver v_{bl}\|_{L^\infty} = O(\delta^{-2}) = O(\varepsilon^{-2s})$. Combining these singular factors, we obtain 
\begin{align}\label{est:R_bl_J_nabla_mu_start}
    \left| \varepsilon^2\int_\Omega J\cdot \nabla \mu'(\rho)\diver v_{bl} \,dx \right| &\leq C \varepsilon^2 \norm{\diver v_{bl}}_{L^\infty} \norm{\Lambda}_{L^2} \norm{\nabla\beta(\rho)}_{L^2}\leq C \varepsilon^2 \delta^{-2}= O(\varepsilon^{1-2s}).
\end{align}
Finally, we consider the last remainder term of $\mathcal{R}_{bl}$,
\begin{equation}\label{est:R_bl_aux_momentum_def}
    \varepsilon^2 \int_\Omega \Big( m^j u^{E,k}\d_k v^j_{bl} - m^j v^k_{bl}\d_k u^{E,j} - m^j u^{E,j}\d_k v^k_{bl} \Big) \, dx.
\end{equation}
Notice that the most singular contributions in \eqref{est:R_bl_aux_momentum_def} arise when the spatial derivative acts on the corrector, yielding $\|\nabla v_{bl}\|_{L^\infty} = O(\delta^{-2}) = O(\varepsilon^{-2s})$. Bounding the Euler velocity $\|u^E\|_{L^\infty} \leq C$, and recalling that $m = \sqrt{\rho}\nabla\beta(\rho)$, we obtain
\begin{align}\label{est:R_bl_aux_momentum_start}
    \left| \varepsilon^2 \int_\Omega m^j u^{E,k}\d_k v^j_{bl} \,dx \right| 
    &\leq C \varepsilon^2 \norm{u^E}_{L^\infty} \norm{\nabla v_{bl}}_{L^\infty} \norm{\sqrt{\rho}}_{L^2} \norm{\nabla\beta(\rho)}_{L^2} \leq C \varepsilon \delta^{-2}= O(\varepsilon^{1-2s}).
\end{align}
The remaining term involving $\nabla u^E$ contains no derivatives on the corrector, scaling as $O(\varepsilon^2 \delta^{-1} \varepsilon^{-1}) = O(\varepsilon^{1-s})$. Then, we conclude that
\begin{equation}\label{est:R_bl_aux_momentum_final}
    \left| \varepsilon^2 \int_\Omega \Big( m^j u^{E,k}\d_k v^j_{bl} - m^j v^k_{bl}\d_k u^{E,j} - m^j u^{E,j}\d_k v^k_{bl} \Big) \, dx \right| = O(\varepsilon^{1-2s}).
\end{equation}
We must now ensure that all exponents of $O(\varepsilon)$ obtained above are strictly positive. Collecting the derived asymptotic rates and taking the intersection, we deduce
\begin{equation}\label{eq:final_s_conditions_1_2}
    s < \min\left\{ \frac{1}{2}, \, \frac{5+\alpha}{3(3+\alpha)} \right\} \quad \text{for } d=1, \qquad
    s < \min\left\{ \frac{1}{2}, \, \frac{3+\alpha}{3(2+\alpha)} \right\} \quad \text{for } d=2,
\end{equation}
and for general dimensions $d \geq 3$,
\begin{equation}\label{eq:final_s_conditions_3d}
    s < \min\left\{ \frac{1}{2}, \, \frac{d(1+\alpha)+4}{3(d(1+\alpha)+2)} \right\}.
\end{equation}
Under these restrictions, the boundary layer remainder $\mathcal{R}_{bl}$ is strictly controlled by the relative energy up to a vanishing error.
\end{proof}

\subsection{Proof of Theorem \ref{main high}}\label{sec main high}
\begin{proof}
    The conclusion of the theorem is a direct consequence of the relative entropy inequality, Lemma \ref{lem:rel1 high}, combined with the uniform bound on the remainder terms established in Proposition \ref{prop high R}.
    
    First, we bridge the gap between the boundary layer approximation and the Euler solution. We decompose the capillary relative energy square,
    \begin{align*}
        \varepsilon^2\int_\Omega |\nabla\beta(\rho^\varepsilon)-\sqrt{\rho^\varepsilon}v^E|^2 \,dx &\leq 2\varepsilon^2 \int_\Omega |\nabla\beta(\rho^\varepsilon)-\sqrt{\rho^\varepsilon}v^E_{bl}|^2 \,dx + 2\varepsilon^2\int_\Omega |\sqrt{\rho^\varepsilon}v_{bl}|^2 \,dx\\
        &\leq 2\varepsilon^2 \int_\Omega |\nabla\beta(\rho^\varepsilon)-\sqrt{\rho^\varepsilon}v^E_{bl}|^2 \,dx + O(\varepsilon^{2-2s}).  
    \end{align*}
    Taking advantage of this splitting, we bound the relative energy by the boundary layer relative energy as follows:
    \begin{equation}\label{eq:final_energy_split}
        \mathcal{E}(\rho^\varepsilon,\Lambda^\varepsilon| \rho^E , u^E, v^E )(\tau) \leq 2 \mathcal{E}(\rho^\varepsilon,\Lambda^\varepsilon| \rho^E , u^E, v^E_{bl} )(\tau) + O(\varepsilon^{2-2s}).
    \end{equation}
    Now, using the remainder bounds from Proposition \ref{prop high R} into the relative energy inequality \eqref{high relative entropy} and applying Gronwall's Lemma, we obtain for almost every $\tau\in (0,T)$:
    \begin{equation}\label{eq:gronwall_proof}
        \mathcal{E}(\rho^\varepsilon,\Lambda^\varepsilon| \rho^E , u^E, v^E_{bl} )(\tau) \leq \Big( \mathcal{E}(\rho^\varepsilon_0,\Lambda^\varepsilon_0| \rho^E_0 , u^E_0, v^E_{bl}(0) ) + O(\varepsilon^\kappa) \Big) e^{C\tau},
    \end{equation}
    where $\kappa = \min\{p, s\} > 0$. Applying the same splitting as in \eqref{eq:final_energy_split} to the initial data, we get
    \begin{equation*}
        \mathcal{E}(\rho^\varepsilon_0,\Lambda^\varepsilon_0| \rho^E_0 , u^E_0, v^E_{bl}(0) ) \leq 2\mathcal{E}(\rho^\varepsilon_0,\Lambda^\varepsilon_0| \rho^E_0 , u^E_0, v^E_0 ) + O(\varepsilon^{2-2s}).
    \end{equation*}
    Substituting this initial bound into \eqref{eq:gronwall_proof} and then into \eqref{eq:final_energy_split}, and noting that $O(\varepsilon^{2-2s})$ is absorbed by $O(\varepsilon^\kappa)$, we deduce
    \begin{equation}
         \mathcal{E}(\rho^\varepsilon,\Lambda^\varepsilon| \rho^E , u^E, v^E )(\tau) \leq C \Big( \mathcal{E}(\rho^\varepsilon_0,\Lambda^\varepsilon_0| \rho^E_0 , u^E_0, v^E_0 ) + O(\varepsilon^\kappa) \Big) e^{C\tau}.
    \end{equation}
    Recalling \eqref{eq:high_ini_data}, the initial data are such that the initial relative energy $\mathcal{E}(\rho^\varepsilon_0,\Lambda^\varepsilon_0| \rho^E_0 , u^E_0, v^E_0 )$ vanishes as $\varepsilon \to 0$. Therefore, we conclude that the relative energy $\mathcal{E}(\rho^\varepsilon,\Lambda^\varepsilon| \rho^E , u^E, v^E )$ goes to zero. Finally, taking advantage of Lemma \ref{lem:convergenza 1_gamma}, we derive the strong convergences
    \begin{multline*}
        \norm{\rho^\varepsilon - \rho^E}_{L^\infty([0,T]; L^1 \cap L^\gamma(\Omega))} 
        + \norm{\Lambda^\varepsilon - \sqrt{\rho^\varepsilon} u^E}_{L^\infty([0,T]; L^2(\Omega))} \\
        + \varepsilon\norm{\nabla\beta(\rho^\varepsilon) - \sqrt{\rho^\varepsilon} v^E}_{L^\infty([0,T]; L^2(\Omega))} \longrightarrow 0 \quad \text{as } \varepsilon \to 0,
    \end{multline*}
    completing the proof of Theorem \ref{main high}.
\end{proof}

\section{Appendix}
\subsection{Extending the Admissible Class of Test Functions}

Using standard density arguments, we can considerably extend the class of admissible test functions $(r, U)$ appearing in the relative energy inequality \eqref{relative entropy}. Indeed, for the left-hand side to be well-defined and the relative energy to be finite, the functions must belong at least to the class $r \in L^\gamma([0,T]\times\Omega)$ and $U \in L^{\frac{2\gamma}{\gamma-1}}([0,T]\times\Omega)$.

Furthermore, a short inspection of the remainder $\mathcal{R}(\rho,\Lambda, r, U)$ implies that all integrals are well-defined provided the test functions satisfy, at least, the following regularity conditions:
\begin{align}
    &\partial_t U \in L^{\frac{2\gamma}{\gamma-1}}([0,T]\times\Omega), \\
    &\nabla U \in L^\infty([0,T]\times\Omega), \quad \diver U \in L^\infty([0,T]\times\Omega), \\
    &\nabla(\diver U) \in 
    \begin{cases}
        L^{\frac{2\gamma}{\gamma-\alpha-2}}([0,T]\times\Omega) & \text{if } \gamma > \alpha+2, \\
        L^\infty([0,T]\times\Omega) & \text{otherwise},
    \end{cases} \\
    &\partial_t f'(r) \in L^{\frac{\gamma}{\gamma-1}}([0,T]\times\Omega), \\
    &\nabla f'(r) \in L^{\frac{2\gamma}{\gamma-1}}([0,T]\times\Omega).
\end{align}
Similarly, when considering the augmented relative energy framework \eqref{eq:aug_rel_b}, we have some additional constraints. A direct application of Hölder's inequality reveals that the augmented formulation is well-defined if the auxiliary velocity field $V$ belongs to the class:
\begin{align}
    &V \in L^{\frac{2\gamma}{\gamma-1}}([0,T]\times\Omega), \quad \partial_t V \in L^{\frac{2\gamma}{\gamma-1}}([0,T]\times\Omega), \\
    &\nabla V \in L^\infty([0,T]\times\Omega), \quad \diver V \in L^\infty([0,T]\times\Omega), \\
    &\nabla(\diver V) \in L^\infty([0,T]\times\Omega).
\end{align}
Consequently, the relative energy inequalities hold for any test functions and auxiliary velocities satisfying the above minimal integrability requirements.

\textbf{Funding.} The authors are partially supported by the PRIN 2022 Project
2022YXWSLR "Boundary analysis for dispersive and viscous fluids”. The first author also acknowledges partial support from Istituto Nazionale di Alta Matematica through the GNAMPA Research Group, by the Italian Ministry of University and Research (MUR)
through the Excellence Department Project awarded to GSSI, CUP D13C22003740001.
Most of this work was done while the second author was at GSSI.

\textbf{Conflict of interest.} The authors declare that they have no conflict of interest.

\textbf{Data availability.} Data sharing is not applicable.
We do not analyse or generate any datasets, because our work proceeds within a theoretical and mathematical approach.

\bibliographystyle{plain}
\bibliography{biblio_kato_qhd}

\end{document}